\documentclass[centertags,12pt]{amsart}
\usepackage{latexsym}
\usepackage{amsthm}
\usepackage{amssymb}
\usepackage{color}
\usepackage[utf8]{inputenc}
\usepackage{mathrsfs}
\usepackage[english]{babel}
\usepackage{setspace}
\usepackage{ textcomp }
\usepackage{graphicx}
\graphicspath{ {images/} }

\numberwithin{equation}{section}
\makeatletter
\renewcommand{\subsection}{\@startsection
{subsection}{2}{0mm}{\baselineskip}{-0.25cm}
{\normalfont\normalsize\bf}}
\makeatother

\newtheorem{theorem}{Theorem}[section]
\newtheorem{proposition}[theorem]{Proposition}
\newtheorem{lemma}[theorem]{Lemma}
\newtheorem{corollary}[theorem]{Corollary}

   {\theoremstyle{definition}

}
   \theoremstyle{remark}
\newtheorem{remark}[theorem]{Remark}

\newcommand{\F}{{\mathbb F}}
\newcommand{\fq}{{\mathbb F_q}}
\newcommand{\fqs}{{\mathbb F_{q^2}}}
\newcommand{\PGU}{{\rm PGU}}
\newcommand{\PSU}{{\rm PSU}}
\newcommand{\PGL}{{\rm PGL}}
\newcommand{\PG}{{\rm PG}}

\newcommand{\cH}{{\mathcal H}}

\newcommand{\aut}{{\rm Aut}}
\newcommand{\ord}{{\rm ord}}
\newcommand{\diag}{{\rm diag}}
\newcommand{\la}{{\langle}}
\newcommand{\ra}{{\rangle}}
\newcommand{\T}{{\times}}
\newcommand{\RT}{{\rtimes}}
\newcommand{\cMq}{{\mathcal{M}_q}}
\newcommand{\bq}{{\bar{q}}}
\newcommand{\cM}{{\mathcal M}}

\begin{document}

\author[ M. Montanucci - G. Zini]{Maria Montanucci and Giovanni Zini}
\title{On the spectrum of genera of quotients of the Hermitian curve
}

\begin{abstract}
We investigate the genera of quotient curves $\mathcal H_q/G$ of the $\mathbb F_{q^2}$-maximal Hermitian curve $\mathcal H_q$, where $G$ is contained in the maximal subgroup $\mathcal M_q\leq{\rm Aut}(\mathcal H_q)$ fixing a pole-polar pair $(P,\ell)$ with respect to the unitary polarity associated with $\mathcal H_q$.
To this aim, a geometric and group-theoretical description of $\mathcal M_q$ is given.
The genera of some other quotients $\mathcal H_q/G$ with $G\not\leq\mathcal M_q$ are also computed.
Thus we obtain new values in the spectrum of genera of $\mathbb F_{q^2}$-maximal curves.
A plane model for $\mathcal H_q/G$ is obtained when $G$ is cyclic of order $p\cdot d$, with $d$ a divisor of $q+1$.
\end{abstract}

\maketitle

\begin{small}



\end{small}

\section{Introduction}

Let $q$ be a power of a prime $p$, $\mathbb F_{q^2}$ be the finite field with $q^2$ elements, and $\mathcal X$ be an $\mathbb F_{q^2}$-rational curve, i.e. a projective, absolutely irreducible, non-singular algebraic curve defined over $\mathbb F_{q^2}$.
The curve $\mathcal X$ is called $\mathbb F_{q^2}$-maximal if the number $|\mathcal X(\mathbb F_{q^2})|$ of its $\fqs$-rational points attains the Hasse-Weil upper bound $q^2+1+2gq$, where $g$ is the genus of $\mathcal X$.
Maximal curves have interesting properties and  have also been investigated for their applications in Coding Theory. Surveys on maximal curves are found in \cite[Chapter 10]{HKT}.

The most important example of an $\fqs$-maximal curve is the Hermitian curve $\cH_q$, defined as any $\fqs$-rational curve projectively equivalent to the plane curve with Fermat equation $X^{q+1}+Y^{q+1}+Z^{q+1}=0$.
For fixed $q$, $\cH_q$ has the largest possible genus $q(q-1)/2$ that an $\fqs$-maximal curve can have.
The full automorphism group $\aut(\cH_q)$ is isomorphic to $PGU(3,q)$, the group of projectivities of $PG(2,q^2)$ commuting with the unitary polarity associated with $\cH_q$.

By a result commonly attributed to Serre, any $\fqs$-rational curve which is $\fqs$-covered by an $\fqs$-maximal curve is also $\fqs$-maximal.
In particular, $\fqs$-maximal curves are given by the Galois $\fqs$-subcovers of an $\fqs$-maximal curve $\mathcal X$, that is by the quotient curves $\mathcal X/G$ over a finite automorphism group $G\leq\aut(\mathcal X)$.

A challenging open problem is the determination of the spectrum $\Gamma(q^2)$ of genera of $\fqs$-maximal curves, for a given $q$.
Most of the known values in $\Gamma(q^2)$ have been obtained from quotient curves $\cH_q/G$ of the Hermitian curve, which have been investigated in many papers; the most significant cases are the following:
\begin{itemize}
\item $G$ fixes an $\fqs$-rational point of $\cH_q$; see \cite{AQ,GSX,BMXY}.
\item $G$ normalizes a Singer subgroup of $\cH_q$ acting on three $\mathbb F_{q^6}$-rational points of $\cH_q$; see \cite{CKT1,GSX}.
\item $G$ has prime order; see \cite{CKT2}.
\item $G$ fixes an $\fqs$-rational point off $\cH_q$ and is isomorphic to a subgroup of $SL(2,q)$; see \cite{CKT2}.
\end{itemize}
Let $\mathcal M_q$ be the maximal subgroup of $PGU(3,q)$ fixing an $\fqs$-rational point $P\notin\cH_q$; equivalently, $\cM_q$ fixes the polar line $\ell$ of $P$ with respect to the unitary polarity associated with $\cH_q$.
The group-theoretical structure of $\cM_q$ is not known, and only few genera of quotients $\cH_q/G$ with $G\leq\cM_q$ have been computed; see \cite{CKT2,GSX}.
In this paper, many genera of quotients $\cH_q/G$ with $G\leq \cM_q$ are computed, giving a partial answer to the question raised by Garcia, Stichtenoth, and Xing in \cite[Remark 6.6]{GSX}.
To this aim, we provide a new geometric and group-theoretical description of $\cM_q$. This allows us to use the techniques developed in \cite{MZRS} for the determination of the genera of quotients $\cH_q/G$.
Some further genera of quotients $\cH_q/G$ with $G\not\leq\cMq$ are also computed.
Our results provide new values in the spectra $\Gamma(q^2)$.
Finally, we consider the quotients $\cH_q/G$ where $G$ is cyclic, and provide a plane model for $\cH_q/G$ in one of the few cases for which an equation is not known, namely when $G$ has order $p\cdot d$ with $d$ a divisor of $q+1$.

\section{Preliminary results}

%\begin{itemize}
%\item Preliminari che stanno su Ree e Suzuki: massimali di PGU,...
%\item Lemma 2.2 e Teorema 2.7 da Ree e Suzuki
%\end{itemize}
Throughout this paper, $q=p^n$, where $p$ is a prime number and $n$ is a positive integer. The Deligne-Lusztig curves defined over a finite field $\mathbb{F}_q$ were originally introduced in \cite{DL}. Other than the projective line, there are three families of Deligne-Lusztig curves, named Hermitian curves, Suzuki curves and Ree curves. The Hermitian curve $\mathcal H_q$ arises from the algebraic group $^2A_2(q)={\rm PGU}(3,q)$ of order $(q^3+1)q^3(q^2-1)$. It has genus $q(q-1)/2$ and is $\mathbb F_{q^2}$-maximal. Thus curve is isomorphic to the curves listed below:
\begin{equation} \label{M1}
X^{q+1}+Y^{q+1}+Z^{q+1}=0;
\end{equation}
\begin{equation} \label{M2}
X^{q}Z+XZ^{q}-Y^{q+1}=0;
\end{equation}
\begin{equation} \label{M3}
XY^{q}-YX^{q}+\omega Z^{q+1}=0,
\end{equation}
where $\omega$ is a fixed element of $\mathbb{K}$ such that $\omega^{q-1}=-1$;
\begin{equation} \label{M4}
XY^{q}+YZ^{q}+\omega ZX^{q}=0.
\end{equation}
Each of the models (\ref{M1}),(\ref{M2}) and (\ref{M3}) is $\mathbb{K}$-isomorphic to $\cH_q$, while the model (\ref{M4}) is $\mathbb{F}_{q^3}$-isomorphic to $\cH_q$, since for a suitable element $a \in \mathbb{F}_{q^6}$, the projective map 
$$k: \mathbb{P}^2(\mathbb{K}) \rightarrow \mathbb{P}^2(\mathbb{K}), \ (x:y:z) \mapsto (ax+y+a^{q^2+1}z: a^{q^2+1}x+ay+z:x+a^{q^2+1}y+az),$$
changes  (\ref{M1}) into (\ref{M4}), see \cite{CKT1}[Proposition 4.6].
The automorphism group $\aut(\cH_q)$ is isomorphic to the projective unitary group $\PGU(3,q)$, and it acts on the set $\cH_q(\mathbb{F}_{q^2})$ of all $\mathbb{F}_{q^2}$-rational points of $\cH_q$ as $\PGU(3,q)$ in its usual $2$-transitive permutation representation.
The combinatorial properties of $\cH_q(\mathbb{F}_{q^2})$ can be found in \cite{HP}. The size of $\cH_q(\mathbb{F}_{q^2})$ is equal to $q^3+1$, and a line of $PG(2,q^2)$ has either $1$ or $q+1$ common points with $\cH_q(\mathbb{F}_{q^2})$, that is, it is either a $1$-secant or a chord of $\cH_q(\mathbb{F}_{q^2})$. Furthermore, a unitary polarity is associated with $\cH_q(\mathbb{F}_{q^2})$ whose isotropic points are those of $\cH_q(\mathbb{F}_{q^2})$ and isotropic lines are the $1$-secants of $\cH_q(\mathbb{F}_{q^2})$, that is, the tangents to $\cH_q$ at the points of $\cH_q(\mathbb{F}_{q^2})$.

From Group theory we need the classification of all maximal subgroups of the projective special subgroup $\PSU(3,q)$ of $\PGU(3,q)$, going back to Mitchell and Hartley; see \cite{M}, \cite{H}, \cite{HO}; and also the classification of all subgroups of $SL(2,q)$, see \cite[Theorem 6.17]{Suzuki}.

\begin{theorem} \label{Mit} Let $d={\rm gcd}(3,q+1)$. Up to conjugacy, the subgroups below give a complete list of maximal subgroups of $\PSU(3,q)$.

\begin{itemize}
\item[(i)] the stabilizer of an $\F_{q^2}$-rational point of $\cH_q$. It has order $q^3(q^2-1)/d$;
\item[(ii)] the stabilizer of an $\F_{q^2}$-rational point off $\cH_q$ $($equivalently the stabilizer of a chord of $\cH_q(\mathbb{F}_{q^2}))$. It has order $q(q-1)(q+1)^2/d$;
\item[(iii)] the stabilizer of a self-polar triangle with respect to the unitary polarity associated to $\cH_q(\mathbb{F}_{q^2})$. It has order $6(q+1)^2/d$;
\item[(iv)] the normalizer of a (cyclic) Singer subgroup. It has order $3(q^2-q+1)/d$ and preserves a triangle
in $\PG(2,q^6)\setminus\PG(2,q^2)$ left invariant by the Frobenius collineation $\Phi_{q^2}:(X,Y,T)\mapsto (X^{q^2},Y^{q^2},T^{q^2})$ of $\PG(2,\bar{\mathbb{F}}_{q})$;

{\rm for $p>2$:}
\item[(v)] ${\rm PGL}(2,q)$ preserving a conic;
\item[(vi)] $\PSU(3,p^m)$ with $m\mid n$ and $n/m$ odd;
\item[(vii)] subgroups containing $\PSU(3,p^m)$ as a normal subgroup of index $3$, when $m\mid n$, $n/m$ is odd, and $3$ divides both $n/m$ and $q+1$;
\item[(viii)] the Hessian groups of order $216$ when $9\mid(q+1)$, and of order $72$ and $36$ when $3\mid(q+1)$;
\item[(ix)] ${\rm PSL(2,7)}$ when $p=7$ or $-7$ is not a square in $\mathbb{F}_q$;
\item[(x)] the alternating group $Alt(6)$ when either $p=3$ and $n$ is even, or $5$ is a square in $\mathbb{F}_q$ but $\mathbb{F}_q$ contains no cube root of unity;
\item[(xi)] the symmetric group $Sym(6)$ when $p=5$ and $n$ is odd;
\item[(xii)] the alternating group $Alt(7)$ when $p=5$ and $n$ is odd;

{\rm for $p=2$:}
\item[(xiii)] $\PSU(3,2^m)$ with $m\mid n$ and $n/m$ an odd prime;
\item[(xiv)] subgroups containing $\PSU(3,2^m)$ as a normal subgroup of index $3$, when $n=3m$ with $m$ odd;
\item[(xv)] a group of order $36$ when $n=1$.
\end{itemize}
\end{theorem}

\begin{theorem} \label{sl2q}
Let $q=p^h$, where $p$ is a prime and $h \in \mathbb{N}$. Up to isomorphism, the subgroups below give a complete list of subgroups of $SL(2,q)$. \\ \\
$\bullet$ Tame subgroups:
\begin{enumerate}
\item $SL(2,5)$, when $p \geq 7$ and $q \equiv 1\  mod \ 5$;
\item $\hat \Sigma_4$, the representation group of $S_4$ in which the transpositions corerspond to the elements of order $4$, for $p \geq 5$ and $q^2 \ \equiv 1 \ mod \ 16$; this group is isomorphic to $SmallGroup(48,28)$. 
\item $SL(2,3)$, when $p \geq 5$;
\item $C_d$, where $e \mid q \pm 1$;
\item A dicyclic group $Dic(n)=\langle a,x \mid a^{2n}=1, \ x^2=a^n, \ x^{-1}ax=a^{-1}\rangle$ of order $4n$, where $n \mid \frac{q \pm 1}{2}$;
\end{enumerate}
$\bullet$ Non-tame subgroups:
\begin{enumerate}
\item $E_{p^k} \rtimes C_d$, where $d \mid q-1$, and $E_{p^k}$ is elementary abelian of order $p^k$;
\item $SL(2,5)$, when $q=3^h$, and $5 \mid q^2-1$;
\item $SL(2,p^k)$, where $k \mid h$;
\item $T L (2,p^k) \cong \langle SL(2,p^k), d_\pi \rangle$ where,
$$d_\pi = \begin{pmatrix}
    w & 0 & 0  \\
    0 & w^{-1} & 0 \\
    0 & 0 & 1
  \end{pmatrix},$$
for $w=\xi^{\frac{p^k+1}{2}}$ and $\mathbb{F}^*_{p^{2k}}= \langle \xi \rangle$.
\end{enumerate}
\end{theorem}
In our investigation it is useful to know how an element of $\PGU(3,q)$ of a given order acts on $\PG(2,\bar{\mathbb{F}}_{q})$, and in particular on $\cH_q(\mathbb{F}_{q^2})$. This can be obtained as a corollary of Theorem \ref{Mit}, and is stated in Lemma $2.2$ with the usual terminology of collineations of projective planes; see \cite{HP}. In particular, a linear collineation $\sigma$ of $\PG(2,\bar{\mathbb{F}}_q)$ is a $(P,\ell)$-\emph{perspectivity}, if $\sigma$ preserves  each line through the point $P$ (the \emph{center} of $\sigma$), and fixes each point on the line $\ell$ (the \emph{axis} of $\sigma$). A $(P,\ell)$-perspectivity is either an \emph{elation} or a \emph{homology} according as $P\in \ell$ or $P\notin\ell$. A $(P,\ell)$-perspectivity is in  $\PGL(3,q^2)$ if and only if its center and its axis are in $\PG(2,\mathbb{F}_{q^2})$. This classification results was obtained in \cite{MZRS}.
\begin{lemma}\label{classificazione}
For a nontrivial element $\sigma\in\PGU(3,q)$, one of the following cases holds.
\begin{itemize}
\item[(A)] ${\rm ord}(\sigma)\mid(q+1)$. Moreover, $\sigma$ is a homology whose center $P$ is a point off $\cH_q$ and whose axis $\ell$ is a chord of $\cH_q(\mathbb{F}_{q^2})$ such that $(P,\ell)$ is a pole-polar pair with respect to the unitary polarity associated to $\cH_q(\mathbb{F}_{q^2})$.
\item[(B)] ${\rm ord}(\sigma)$ is coprime with $p$. Moreover, $\sigma$ fixes the vertices $P_1,P_2,P_3$ of a non-degenerate triangle $T$.
\begin{itemize}
\item[(B1)] The points $P_1,P_2,P_3$ are $\fqs$-rational, $P_1,P_2,P_3\notin\cH_q$ and the triangle $T$ is self-polar with respect to the unitary polarity associated to $\cH_q(\mathbb{F}_{q^2})$. Also, $\ord(\sigma)\mid(q+1)$.
\item[(B2)] The points $P_1,P_2,P_3$ are $\fqs$-rational, $P_1\notin\cH_q$, $P_2,P_3\in\cH_q$. %and the polar lines of $P_1,P_2,P_3$ are $P_2P_3$, $P_1P_2$, $P_1P_3$, respectively.
     Also, $\ord(\sigma)\mid(q^2-1)$ and $\ord(\sigma)\nmid(q+1)$.
\item[(B3)] The points $P_1,P_2,P_3$ have coordinates in $\F_{q^6}\setminus\F_{q^2}$, $P_1,P_2,P_3\in\cH_q$. %and their polar lines are $P_1P_2$, $P_2P_3$, $P_3P_1$, respectively.
    Also, $\ord(\sigma)\mid (q^2-q+1)$.
\end{itemize}
\item[(C)] ${\rm ord}(\sigma)=p$. Moreover, $\sigma$ is an elation whose center $P$ is a point of $\cH_q$ and whose axis $\ell$ is a tangent of $\cH_q(\mathbb{F}_{q^2})$ such that $(P,\ell)$ is a pole-polar pair with respect to the unitary polarity associated to $\cH_q(\mathbb{F}_{q^2})$.
\item[(D)] ${\rm ord}(\sigma)=p$ with $p\ne2$, or ${\rm ord}(\sigma)=4$ and $p=2$. Moreover, $\sigma$ fixes an $\fqs$-rational point $P$, with $P \in \cH_q$, and a line $\ell$ which is a tangent  of $\cH_q(\mathbb{F}_{q^2})$, such that $(P,\ell)$ is a pole-polar pair with respect to the unitary polarity associated to $\cH_q(\mathbb{F}_{q^2})$.
\item[(E)] $p\mid{\rm ord}(\sigma)$, $p^2\nmid{\rm ord}(\sigma)$, and ${\rm ord}(\sigma)\ne p$. Moreover, $\sigma$ fixes two $\fqs$-rational points $P,Q$, %the line $PQ$ linewise, and another line $\ell$ through $P$ linewise.
     with $P\in\cH_q$, $Q\notin\cH_q$. %$PQ$ is the polar line of $P$, and $\ell$ is the polar line of $Q$.
\end{itemize}
\end{lemma}

Throughout the paper, a nontrivial element of $\PGU(3,q)$ is said to be of type (A), (B), (B1), (B2), (B3), (C), (D), or (E), as given in Lemma \ref{classificazione}. Moreover, $G$ always stands for a subgroup of $\PGU(3,q)$.

From Function field theory we need the Riemann-Hurwitz formula; see \cite[Theorem 3.4.13]{Sti}. Every subgroup $G$ of $\PGU(3,q)$ produces a quotient curve $\cH_q/G$, and the cover $\cH_q\rightarrow\cH_q/G$ is a Galois cover defined over $\mathbb{F}_{q^2}$  where the degree of the different divisor $\Delta$ is given by the Riemann-Hurwitz formula, namely
$\Delta=(2g(\cH_q)-2)-|G|(2g(\cH_q/G)-2)$. On the other hand, $\Delta=\sum_{\sigma\in G\setminus\{id\}}i(\sigma)$, where $i(\sigma)\geq0$ is given by the Hilbert's different formula \cite[Thm. 3.8.7]{Sti}, namely
\begin{equation}\label{contributo}
\textstyle{i(\sigma)=\sum_{P\in\cH_q(\bar\F_q)}v_P(\sigma(t)-t),}
\end{equation}
where $t$ is a local parameter at $P$.

By analyzing the geometric properties of the elements $\sigma \in \PGU(3,q)$, it turns out that there are only a few possibilities for $i(\sigma)$.
This is obtained as a corollary of Lemma \ref{classificazione} and stated in the following proposition, see \cite{MZRS}.

\begin{theorem}\label{caratteri}
For a nontrivial element $\sigma\in \PGU(3,q)$ one of the following cases occurs.
\begin{enumerate}
\item If $\ord(\sigma)=2$ and $2\mid(q+1)$, then $\sigma$ is of type {\rm(A)} and $i(\sigma)=q+1$.
\item If $\ord(\sigma)=3$, $3 \mid(q+1)$ and $\sigma$ is of type {\rm(B3)}, then $i(\sigma)=3$.
\item If $\ord(\sigma)\ne 2$, $\ord(\sigma)\mid(q+1)$ and $\sigma$ is of type {\rm(A)}, then $i(\sigma)=q+1$.
\item If $\ord(\sigma)\ne 2$, $\ord(\sigma)\mid(q+1)$ and $\sigma$ is of type {\rm(B1)}, then $i(\sigma)=0$.
\item If $\ord(\sigma)\mid(q^2-1)$ and $\ord(\sigma)\nmid(q+1)$, then $\sigma$ is of type {\rm(B2)} and $i(\sigma)=2$.
\item If $\ord(\sigma)\ne3$ and $\ord(\sigma)\mid(q^2-q+1)$, then $\sigma$ is of type {\rm(B3)} and $i(\sigma)=3$.
\item If $p=2$ and $\ord(\sigma)=4$, then $\sigma$ is of type {\rm(D)} and $i(\sigma)=2$.
\item If $\ord(\sigma)=p$, $p \ne2$ and $\sigma$ is of type {\rm(D)}, then $i(\sigma)=2$.
\item If $\ord(\sigma)=p$ and $\sigma$ is of type {\rm(C)}, then $i(\sigma)=q+2$.
\item If $\ord(\sigma)\ne p$, $p\mid\ord(\sigma)$ and $\ord(\sigma)\ne4$, then $\sigma$ is of type {\rm(E)} and $i(\sigma)=1$.
\end{enumerate}
\end{theorem}

\section{The maximal subgroup $\cMq$ of $\PGU(3,q)$ for $q$ odd} \label{secmassimale}

The aim of this section is to give an explicit description of the maximal subgroup $\cMq$ of $\PGU(3,q)$. Geometrically, $\cMq$ preserves a non-tangent point-line pair $(P_0,\ell)$, where $P_0 \in \mathbb{P}^2(\fqs) \setminus \cH_q$ and $\ell$ is its polar line with respect to the unitary polarity associated to $\cH_q$. The line $\ell$ is a $\mathbb{F}_{q^2}$-rational line meeting $\cH_q$ in $q+1$ pairwise distinct $\mathbb{F}_{q^2}$-rational points.
We use the plane model (\ref{M3}) of $\cH_q$. Up to conjugation we can choose $P_0=(0:0:1)$ and $\ell:Z=0$. Then the subgroup of automorphisms of $\mathbb{P}^2(\mathbb{K})$ preserving both $(P_0,\ell)$ and $\cH_q$, consists of maps of type
\begin{equation} \label{map}
(X,Y,Z) \mapsto (aX+bY,cX+dY,Z),
\end{equation} 
where
$$ac^q-a^qc=0, \ bd^q-b^qd=0, \ bc^q-a^qd=-1, \ ad^q-b^qc=1.$$
Clearly, $|\cMq|=q(q-1)(q+1)$. Those maps with $a,b,c,d \in \mathbb{F}_q$ and $ad-bc=1$, form a subgroup $H$ of $\cMq$ which is isomorphic to $SL(2,\fq)$. By direct checking $H$ is the commutator subgroup of $\cMq$ and hence it is normal in $\cMq$. We recall that since $\mathbb{K}$ has odd characteristic, every involution $\tau$ is an homology, see Lemma $2.2$. We can give a geometrical interpretation of $\cMq$ as $\cMq=C_{\PGU(3,q)}(\tau)=N_{\PGU(3,q)}(\tau)$, where $\tau$ is the $(P_0,\ell)$-involution given by
$$\tau : (X,Y,Z) \mapsto (-X,-Y,Z).$$
The center of $\cMq$, say $Z(\cMq)$, is a cyclic group of order $q+1$. This subgroup consists of all the $(P_0,\ell)$-homologies contained in $\cMq$. Summarizing,
$$\cMq= C_{\PGU(3,q)}(\tau)=\Bigg\{ \begin{pmatrix}
    a & b & 0  \\
    c & d & 0 \\
    0 & 0 & 1
  \end{pmatrix}  \mid ac^q-a^qc=0, \ bd^q-b^qd=0, \ bc^q-a^qd=-1, \ ad^q-b^qc=1 \Bigg\},$$ 
$$H=(\cMq)'= \Bigg\{ \begin{pmatrix}
    a & b & 0  \\
    c & d & 0 \\
    0 & 0 & 1
  \end{pmatrix}  \in \cMq \mid a,b,c,d \in \mathbb{F}_q \  and \ ad-bc=1 \Bigg\} \cong SL(2,q),$$
$$Z(\cMq)=\Bigg\{ \begin{pmatrix}
    a & 0 & 0  \\
    0 & a & 0 \\
    0 & 0 & 1
  \end{pmatrix}  \in \cMq \mid a^{q+1}=1 \Bigg\} \cong C_{q+1}.$$

\begin{proposition} \label{massimale} Let $\tau'$ be a fixed $(P'_0,\ell')$-involution, where $P'_0 \in \ell$, $P'_0 \not\in \cH_q$ and $\ell'$ is the polar line of $P'_0$ with respect to the polarity associated to $\cH_q$. Then
$$\cMq = H \rtimes Z(C_{\PGU(3,q)}(\tau')) \cong SL(2,q) \rtimes C_{q+1},$$
where $Z(C_{\PGU(3,q)}(\tau'))$ is the center of the centralizer of $\tau'$ in $\PGU(3,q)$.
\end{proposition}

\begin{proof} Since $C_{\PGU(3,q)}(\tau')$ is a subgroup of $\PGU(3,q)$ which is conjugated to $\cMq$, we know that $Z(C_{\PGU(3,q)}(\tau'))$ is the group of all the $(P'_0,\ell')$-homologies and hence it is cyclic of order $q+1$ and since $Z(C_{\PGU(3,q)}(\tau'))$ fixes $P_0$, we have that $Z(C_{\PGU(3,q)}(\tau')) \subseteq \cMq$. Since $H=(\cMq)'$, the product $HZ(C_{\PGU(3,q)}(\tau'))$ is a subgroup of $\cMq$. Moreover, $Z(C_{\PGU(3,q)}(\tau')) \cap H = \{1\}$ because, looking at the matrix representations, every element of $H$ has determinant equal to $1$. Thus, $HZ(C_{\PGU(3,q)}(\tau')) = H \rtimes Z(C_{\PGU(3,q)}(\tau'))$ and since $|HZ(C_{\PGU(3,q)}(\tau'))|=|H|(q+1)=|\cMq|$, the claim follows.  
\end{proof}

We can give an explicit matrix representation for $\cMq$ as follows. In Proposition \ref{massimale} we prove that a complement for $H$ in $\cMq$ is given by a cyclic group of order $q+1$ given by $(P'_0,\ell')$-homologies, where $P'_0 \in \ell$, $P'_0 \not\in \cH_q$ and $\ell'$ is the polar line of $P'_0$ with respect to the polarity associated to $\cH_q$. We can construct such a complement fixing a collineation $\sigma \in PGU(3,q)$ such that $\sigma(P_0)=P'_0$. Clearly, $Z(C_{\PGU(3,q)}(\tau))^{\sigma}=Z(C_{\PGU(3,q)}(\tau'))$ and $H \cap Z(C_{\PGU(3,q)}(\tau))^{\sigma}=\{1\}$. In this way the following corollary is obtained.

\begin{corollary} \label{explicit}
Let $e$ be a primitive element of $\mathbb{F}_{q^2}$. Then
$$\cMq=H \rtimes \langle \alpha \rangle,$$
where, $\alpha$ is associated to the matrix representation
$$\alpha= \begin{pmatrix}
    0 & e^{-1} & 0  \\
    -e^{q} & 1+e^{q+1} & 0 \\
    0 & 0 & 1
  \end{pmatrix}.$$

\end{corollary}

The following lemma collects geometric properties of the elements of $\cMq$. It is obtained as a direct consequence of Lemma $2.2$ and Theorem $2.7$.

\begin{lemma} \label{pelem} 
Let $\sigma \in \cMq$. Then,
\begin{enumerate}
\item If $\sigma$ is a $p$-element then $\sigma$ is of type $(C)$ and $i(\sigma)=q+2$;
\item If $\sigma \in H$ and $2 \ne o(\sigma) \mid q+1$ then $\sigma$ is of type $(B1)$ and $i(\sigma)=0$.
\end{enumerate}
\end{lemma}

\begin{proof} Let $\sigma \in \cMq$ be a $p$-element. From Lemma $2.2$ $\sigma$ is either of type $(C)$ or $(D)$. Since $\sigma$, by definition of $\cMq$, must fix $P_0 \not\in \cH_q$ then $\sigma$ cannot be of type $(D)$. Now, $(1)$ follows from Theorem $2.7$. Let $\sigma \in \cMq$ with $o(\sigma) \ne 2$ but $o(\sigma) \mid q+1$. Since, looking at the matrix representation of $\sigma$, we have that $det(\sigma)=1$, then
$$\sigma = \begin{pmatrix}
    a & 0 & 0  \\
    0 & a^{-1} & 0 \\
    0 & 0 & 1
  \end{pmatrix},$$
where $o(a)=o(\sigma)$. Thus, by direct checking $\sigma$ is of type $(B1)$, and so $i(\sigma)=0$ from Theorem $2.7$.
\end{proof}

Let $G$ be a subgroups of $H$. The following lemma shows that the action of $G$ on the affine points of $\cH_q$ is semi-regular, i.e. each point-orbit of affine points of $\cH_q$ under $G$ has length equal to the order of $G$, see \cite{CKT1}.

\begin{lemma} Let $\alpha \in H$ and $P \in \cH_q$ an affine point such that $\alpha(P)=P$. Then $\alpha$ is the identity map.
\end{lemma}

\begin{proof} It follows from $(\ref{map})$ and the fact that $a^q=a$ for each $a \in \mathbb{F}_q$.
\end{proof}

Now we investigate the action of $G$ on the set $\mathcal{I}= \ell \cap \cH_q$ consisting of all points $(1:m:0)$, with $m \in \mathbb{F}_q$, together with $(0:1:0)$. Since $H$ acts on $\mathcal{I}$ as $PSL(2,q)$ in its natural $2$-transitive permutation representation on $\mathbb{P}^1 (\mathbb{F}_q)$, we have actually to consider $\bar{G}$ instead of $G$, where $\bar{G}$ is the image of $G$ under the canonical epimorphism
$$\phi: H \cong SL(2,q) \rightarrow PSL(2,q).$$
Note that the kernel of $\phi$ is trivial for $p=2$, otherwise it is the subgroup of order $2$ generated by the involution $\tau$. Hence either $ord(G)=2ord(\bar{G})$ or $ord(G)=ord(\bar{G})$, and in the later case $ord(\bar{G})$ must be odd. \\ The tame subgroups of $H$ are analyzed in \cite{CKT1}, using this equivalent representation. In the following section a large class of tame and non-tame subgroups of $\cMq$ is studied.

\section{Genera of quotient curves $\cH_q / H$, for some $H \leq \cMq$}
%Qui i generi dai sotto-prodotti semidiretti del massimale punto-retta
%Nota che per molti valori di q questi esauriscono il massimale: see Section successiva
In this section our aim is to compute the genus of the quotient curve of $\cH_q$ arising from those subgroups of $\cMq$ that can be written, up to isomorphism, as an internal semidirect product of a subgroup of $H$ and a subgroup of $Z(C_{\PGU(3,q)}(\tau')) \cong C_{q+1}$. To this purpose, we use the same notation and representation for $\cMq$ introduced in Section \ref{secmassimale}.

\subsection{Non-tame subgroups of ${\bf H \cong SL(2,q)}$} \ \\

This section provides the genus of $\cH_q/G$ for all non-tame subgroups of $H$, using the complete classification of subgroups of $SL(2,q)$ given in Theorem \ref{sl2q}.
When $G\leq H$ is tame, the genus of $\cH_q/G$ was already obtained in \cite[Section 3]{CKT2}.

Since $\cH_q$ has zero $p$-rank, each $p$-subgroup has a unique fixed point on $\cH_q$, see \cite[Lemma 11.129]{HKT}. Let $G$ be the subgroup of $H \cong SL(2,q)$ given by $E_{p^k} \rtimes C_{d}$, where $d \mid q-1$. Since $E_{p^k}$ is a normal subgroup of $G$, then $G$ fixes a point $P \in \cH_q(\mathbb{F}_{q^2})$. In this case the computation of the genus of the quotient curve $\cH_q / G$ is computed in \cite{GSX}.

\begin{proposition}{\rm (\cite[Theorem 4.4]{GSX})}
Let $G\leq H$ with $G\cong E_{p^k}\rtimes C_d$ with $d\mid(q-1)$.
Then the genus $\bar g$ of the quotient curve $\cH_q/G$ is
$$ \bar g= \frac{(q-p^k)(q+1-\gcd(d,2))}{2dp^k} $$
\end{proposition}

\begin{proposition} Let $G \leq H$ with $G \cong SL(2,5)$, for $q=3^h$ and $5 \mid q^2-1$. Then the genus $\bar{g}$ of the quotient curve $\cH_q / G$ is given by:
$$\bar{g} =\frac{q^2-22q+117-48[h+2]_{mod \ 4}}{240}.$$
\end{proposition}

\begin{proof} By direct computation we get that $5 \mid q^2-1$ if and only if $q=3^h$ where $h$ is even. We can apply Theorem \ref{caratteri} and Lemma \ref{pelem} looking at the congruence of $h$ modulo $4$; in this way the degree of the different divisor is now obtained simply analyzing the order statistic of the elements of $G$. Assume that $h \equiv 0 \ mod \ 4$. By direct checking $5 \mid q-1$ and so also $10 \mid q-1$. From Theorem \ref{caratteri} we have
$$i(\sigma)= \begin{cases} q+1, \ if \ o(\sigma)=2; \\ q+2, \ if \ o(\sigma)=3; \\ 2, \ \ \ \ \ if \ o(\sigma) \in \{5,4,10\}; \\ 1,\  \ \ \ \ if \ o(\sigma)=6.\end{cases}$$
Using the same argument we get, if $h \equiv 2 \ mod \ 4$:
$$i(\sigma)= \begin{cases} q+1, \ if \ o(\sigma)=2; \\ q+2, \ if \ o(\sigma)=3; \\ 0, \ \ \ \ \ if \ o(\sigma) \in \{5,10\}; \\ 2, \  \ \ \ \ if \ o(\sigma)=4; \\ 1,\  \ \ \ \ if \ o(\sigma)=6.\end{cases}$$
Since $SL(2,5)$ contains exactly $1$ element of order $2$, $20$ elements of order $3$, $24$ elements of order $5$, $30$ elements of order $4$, $20$ elements of order $6$ and $24$ elements of order $10$, the claim follows as a direct application of the Riemann-Hurwitz formula.
\end{proof}

\begin{proposition} \label{sottosl} Let $G \leq H$ with $G \cong SL(2,\bar{q})$, for $\bar{q}=p^k$, $k \leq h$. Then the genus $\bar{g}$ of the quotient curve $\cH_q / G$ is given by:
$$\bar{g} =\begin{cases}
\frac{(q-\bar q)(q-\bar q^2+\bar q-1)}{2\bar{q}(\bar{q}^2-1)}, & \ \textrm{if} \ q=\bar{q}^{2r+1},r\geq0; \vspace*{2 mm} \\ 
\frac{(q-1)(q-\bar q^2)}{2\bar{q}(\bar{q}^2-1)} ,& \ \textrm{if} \ q=\bar{q}^{2r},r\geq1.\end{cases}$$
\end{proposition}

\begin{proof} 
Suppose that $q=\bar{q}$. Looking at the structure of the conjugacy classes of $SL(2,q)$ (see for instance \cite[\S 3.6]{Suzuki}) we get the following classification for the elements of $SL(2,q)$:
\begin{itemize}
\item $1$ element $\alpha$ of order $2$. Thus, $i(\alpha)=q+1$, from Theorem \ref{caratteri};
\item $\frac{q(q-1)^2}{2}$ nontrivial elements whose order is different from $2$ but divides $q+1$. From Lemma \ref{pelem} the contribution to the degree of the different divisors of these elements is equal to $0$;
\item $\frac{q(q+1)(q-3)}{2}$ nontrivial elements whose order is different from $2$ but divides $q-1$. From Theorem \ref{caratteri} the contribution to the degree of the different divisors of these elements is equal to $2$;
\item $(q-1)(q+1)$ elements of order $p$. From Lemma \ref{pelem} the contribution of these elements is equal to $q+2$;
\item $q^2-1$ elements of order a multiple of $p$ different from $p$. From Theorem \ref{caratteri} their contribution to the different divisor is equal to $1$.  
\end{itemize}

Suppose that $q=\bar{q}^{2r+1}$, so that $\bar{q}+1 \mid q+1$ and $\bar{q}-1 \mid q-1$. From the Riemann-Hurwitz formula and Theorem \ref{caratteri} we get that
$$(q+1)(q-2)=$$
$$\bar{q}(\bar{q}+1)(\bar{q}-1)(2\bar{g}-2) +(q+1)+2\frac{\bar{q}(\bar{q}+1)(\bar{q}-3)}{2}+0 \frac{\bar{q}(\bar{q}-1)^2}{2}+(\bar{q}-1)(\bar{q}+1)(q+2)+(\bar{q}^2-1),$$
and the claim follows by direct computation.

Suppose that $q=\bar{q}^{2r}$, so that $\bar{q}-1,\bar{q}+1 \mid q-1$. From the Riemann-Hurwitz formula and Theorem \ref{caratteri} we get that
$$(q+1)(q-2)=$$
$$\bar{q}(\bar{q}+1)(\bar{q}-1)(2\bar{g}-2) +(q+1)+2 \frac{\bar{q}(\bar{q}+1)(\bar{q}-3)}{2} +2 \frac{\bar{q}(\bar{q}-1)^2}{2}+(\bar{q}-1)(\bar{q}+1)(q+2)+(\bar{q}^2-1).$$
\end{proof}

\begin{proposition} Let $G \cong T L (2,p^k)= \langle SL(2,p^k), d_\pi \rangle$ where,
$$d_\pi = \begin{pmatrix}
    w & 0 & 0  \\
    0 & w^{-1} & 0 \\
    0 & 0 & 1
  \end{pmatrix},$$
for $w=\xi^{\frac{p^k+1}{2}}$ and $\mathbb{F}^*_{p^{2k}}= \langle \xi \rangle$.
Then the genus $\bar{g}$ of the quotient curve $\cH_q / G$ is given by:
$$ \bar{g} =\frac{(q-\bar q^2)(q-1)}{4\bar{q}(\bar{q}^2-1)}. $$
\end{proposition}

\begin{proof} Let $\bar{q}=p^k$.
Since $d_\pi\in SL(2,\bar q^2)\setminus SL(2,\bar q)$, $q$ is an even power $\bar q^{2r}$ of $\bar q$.
The values $i(\sigma)$ for $\sigma \in SL(2,\bar{q})$ are computed in Proposition \ref{sottosl}.
Since $o(d_\pi)=2(\bar{q}-1)$, we have that $o(d_\pi) \mid (\bar{q}^2-1)$ and $o(d_\pi) \nmid (\bar{q}+1)$. Thus, according to Theorem \ref{caratteri} we have that $i(d_\pi)=2$.
We need to compute $i(\alpha {d_\pi}^s)$ for every $\alpha \in SL(2,\bar{q})$ and $s \in \mathbb{N}$.
Note that $q$ is an even power $\bar q^{2r}$ of $\bar q$, otherwise $d_{\pi}\notin H$; hence $\bar q^2-1$ divides $q-1$. We show that any element $\alpha d_\pi$ with $\alpha \in SL(2,\bar q)$ is of type (B2).
Suppose $o(\alpha d_\pi)=2$; then $\alpha d_\pi$ is the central involution $\eta$ of $SL(2,\bar q)$, a contradiction to $d_\pi\notin SL(2,\bar q)$.
Suppose $o(\alpha d_\pi)=4$. Then $(\alpha d_\pi)^2=\gamma d_\pi^2$ ($\gamma \in SL(2,\bar q)$) is equal to $\eta$ and hence acts trivially on the line $\ell$; this implies that $\gamma$ is the unique element $(d_\pi^2)^{-1}\in SL(2,\bar q)$ acting as the inverse of $d_\pi^2$ on $\ell$, a contradiction.
Hence $o(\alpha d_\pi)>4$. Since $(\alpha d_\pi)^{2(\bar q-1)}\in SL(2,\bar q)$ is either trivial or of type (B2), the claim follows.
From the Riemann-Hurwitz formula
$$(q+1)(q-2)=2 \bar q(\bar q+1)(\bar q-1)(2 \bar g-2) + (q+1)+2 \frac{\bar{q}(\bar{q}+1)(\bar{q}-3)}{2}$$
$$+2 \frac{\bar{q}(\bar{q}-1)^2}{2}+(\bar{q}-1)(\bar{q}+1)(q+2)+(\bar{q}^2-1) +\bar q(\bar q+1)(\bar q-1) \cdot 2$$
\end{proof}

\subsection{Some tame subgroups of $\cMq$} \ \\ \\
We refer to the complete classification of subgroups of $SL(2,q)$ given in Theorem \ref{sl2q}. In this Section the genus of every quotient curve $\cH_q / G$, for $G=N \rtimes C$ with $N<H$ and $C \cong C_{m}$, $m \mid (q+1)$, is computed.

\begin{proposition} \label{ciccic1} 
Let $G=C_d \rtimes C_m=\langle \beta \rangle \rtimes \langle \alpha \rangle$ where $C_d < H$, $d \mid (q-1)$, and $m \mid (q+1)$ of $\cMq$ then the genus $\bar g$ of the quotient curve $\cH_q / G$ is given by one of the following values:
$$\bar g = \begin{cases} (q-1)^2/(4d), & if \ d \ is \ odd, \ m=2 \ and \ \alpha \beta=\beta \alpha, \\
(q-1)(q-d)/(4d), &   if \ d \ is \ odd, \ m=2 \ and \ \alpha \beta \ne \beta \alpha, \\
(q-1)(q-d-1)/(4d), & if \ d \ is \ even \ and \ m=2,\\
(q-1)(q-m+1)/(2md), & if \ d \ is \ odd, \ m>2 \ and \ \alpha\beta=\beta\alpha, \\
(2(q^2-1)-m(q-1-2d))/(4md), & if \ d \ is \ odd, \ m>2, \ m\equiv_4 0, \ and \ \alpha\beta\ne\beta\alpha, \\
(2(q^2-1-d(q+1))-m(q-1+2d))/(4md), & if \ d \ is \ odd, \ m>2, \ m\equiv_4 2, \ and \ \alpha\beta\ne\beta\alpha, \\
(q-1)(q+1-2m)/(2md), & \ if \ d \ is \ even, \ m>2 \ is \ odd \ and \ \beta\alpha=\alpha\beta,\\
(q-1-d)(q+1-m)/(2md), & \ if \ d \ is \ even, \ m\equiv_4 2 \ and \ \beta\alpha\ne\alpha\beta.
 \end{cases}$$  
\end{proposition}

\begin{proof}
\textbf{Case 1. d  is odd and $ \bf m=2$}
\
\\ 
From Lemma \ref{classificazione} and Theorem \ref{caratteri}, every $\sigma \in C_d$ is of type (B2) and hence $i(\sigma)=2$. Let $P, Q$ and $R$ be the fixed points of $\beta$, with $P,Q \in \cH_q(\mathbb{F}_{q^2})$ and $R \not\in \cH_q$. Assume that $\alpha$ and $\beta$ commute. Since $d$ and $p$ are odd, $G=C_{2d}$ and $2d \mid (q-1)$. Every element $\gamma=\sigma \alpha$ with $\sigma \in C_d$ is of type (B2) as $\gamma$ has exactly the same fixed points of $\sigma$. Proposition \ref{caratteri} and the Riemann-Hurwitz formula yield
$$q^2-q-2=2d(2 \bar g -2) + (q+1) + 2(d-1) +2(d-1),$$
and the claim follows. Assume that $\alpha$ and $\beta$ do not commute. Thus, $\alpha(P)=Q$ and $\alpha(R)=R$. Let $\sigma \in C_d$ with $o(\sigma)=d^\prime \mid d$. Up to conjugation we can assume $P=(0:0:1)$ and $Q=(1:0:0)$ where $\cH_q$ has equation \eqref{M2}. Equivalently up to conjugation, $\sigma$ is given by the following matrix representation,
$$\sigma=\begin{pmatrix} a^{q+1} & 0 & 0 \\ 0 & a & 0 \\ 0 & 0 & 1 \end{pmatrix},$$
for some $a \in \mathbb{F}_{q^2}$ with $o(a)=d^\prime$, and hence
$$\alpha=\begin{pmatrix} 0 & 0 & 1 \\ 0 & 1 & 0 \\ 1 & 0 & 0 \end{pmatrix}.$$
By direct computation $o(\sigma\alpha)=2\frac{o(\sigma)}{(o(\sigma),q-1)}$ and since $d^\prime \mid (q-1)$ we have $o(\sigma \alpha)=2$ for every $\sigma \in C_d$. This proves that $G$ is a dihedral group of order $2d$, and Theorem \ref{caratteri} and the Riemann-Hurwitz formula yield
$$(q^2-q-2)=2d(2 \bar g -2)+2(d-1)+(q+1)+(d-1)(q+1),$$
and the claim follows.
\ \\ 
\textbf{Case 2. d  is even and $ \bf m=2$}
\ \\ 
From Lemma \ref{classificazione}, $C_d$ contains a unique element of order $2$ which is a $(R, \ell_R)$-homology. Thus, $\beta$ and $\alpha$ cannot commute since otherwise $C_d$ and $C_m$ are not disjoint.  Arguing as in Case 1, $G$ is a dihedral group of order $2d$ and hence from the Riemann-Hurwitz formula
$$(q^2-q-2)=2d(2 \bar g-2)+(q+1)+(q+1)+2(d-2)+(d-1)(q+1).$$
\textbf{Case 3. d  is odd and $ \bf m>2$}
\ \\ 
Assume that $\alpha$ and $\beta$ commute; equivalently, assume that $\alpha$ is a $(R,\ell_R)$-homology. From Theorem \ref{classificazione} and the Riemann-Hurwitz formula
$$(q^2-q-2)=md(2 \bar g -2)+(d-1)2+(m-1)(q+1)+(m-1)(d-1)2.$$
Thus, assume that $\beta$ and $\alpha$ do not commute. Hence $\alpha(P)=Q$ and $\alpha(R)=R$. Since $C_m$ has an orbit of length $2$, $m$ must be even and $\alpha^2$ is a $(R,\ell_R)$-homology of order $m/2$. We note that in this case $\alpha$ and the other elements of $C_m$ which lie outside $\langle \alpha^2 \rangle$ must be of type (B1) since $m>2$ and a homology acts semiregularly outside its center and axis. Let $\alpha^\prime \in C_m \setminus \langle \alpha^2 \rangle$ and $\beta^\prime \in C_d$.  As before we can assume up to conjugation that $\beta^\prime$ is given by the following matrix representation,
$$\beta^\prime=\begin{pmatrix} a^{q+1} & 0 & 0 \\ 0 & a & 0 \\ 0 & 0 & 1 \end{pmatrix},$$
for some $a \in \mathbb{F}_{q^2}$ with $o(a)=o(\beta^\prime)$, and hence
$$\alpha^\prime=\begin{pmatrix} 0 & 0 & A \\ 0 & 1 & 0 \\ B & 0 & 0 \end{pmatrix},$$
for some $A,B \in \mathbb{F}_{q^2}^*$. Thus,
$$(\beta^\prime \alpha^\prime)^2 =\begin{pmatrix} AB & 0 & 0 \\ 0 & 1 & 0 \\ 0 & 0 & AB \end{pmatrix}.$$
This prove that $o(\beta^\prime \alpha^\prime)=2o(AB)=o(\alpha^\prime) \mid (q+1)$ and thus $\beta^\prime \alpha^\prime$ is either of type (B1) or of type (A). In particular from the matrix representation of $\beta^\prime \alpha^\prime$ we get that $\beta^\prime \alpha^\prime$ is of type (A) if and only if $o(\beta^\prime \alpha^\prime)=2$ and so if and only if $o(\alpha^\prime)=2$. Now the Riemann-Hurwitz formula yields
$$(q^2-q-2)=md(2g-2)+(d-1)2+(m/2-1)(q+1) +\delta,$$
where $\delta$ satisfies
$$\delta= \begin{cases} (m/2-1)(d-1)2+(m/2-1)(d-1)0+(q+1)d, & if \ m/2 \ is \ odd, \\ (m/2-1)(d-1)2+(m/2)(d-1)0, \ otherwise.\end{cases}$$
\ \\ 
\textbf{Case 4. d  is even and $ \bf m>2$}
\ \\ 
In this case $C_d$ has exactly one element of order $2$ which is of type (A), and $d-2$ nontrivial elements of type (B2). Assume that $\alpha$ and $\beta$ commute. Since $C_d$ and $C_m$ are disjoint this case can happen if and only if $m$ is odd. Arguing as before, from the Riemann-Hurwitz formula 
$$(q^2-q-2)=md(2 \bar g-2)+(d-2)2+(q+1)+(m-1)(q+1)+(m-1)(d-2)2+(m-1)(q+1).$$
Arguing as in Case 3, $m/2$ must be odd and the Riemann-Hurwitz formula yields
$$ q^2-q-2=md(2\bar g-2)+(d-2)2+(q+1)+(m/2-1)(q+1)+(q+1)$$
$$+(m/2-1)(d-2)2+(m/2-1)(q+1)+(q+1)+(d-2)(q+1). $$
\end{proof}
\begin{remark}\label{remoforte}
From the proof of Proposition \ref{ciccic1} follows that the values given for $\bar g$ in Proposition \ref{ciccic1} are exactly all the possible values of $g(\cH_q/G)$ for $G\leq\cMq$, $G\cong C_d\rtimes C_m$, $d\mid(q-1)$, $m\mid(q+1)$.
In fact, a matrix representation for the generators of $G$ can be explicitly provided as follows, where $\cH_q$ has equation \eqref{M2}:
\begin{itemize}
\item if $\alpha\beta=\beta\alpha$, then $\beta=\diag[a^{q+1},a,1]$ and $\alpha=\diag[1,A,1]$;
\item if $\alpha\beta\ne\beta\alpha$, then $\beta=\diag[a^{q+1},a,1]$ and 
$$\alpha=\begin{pmatrix} 0 & 0 & A \\ 0 & 1 & 0 \\ B & 0 & 0 \end{pmatrix}.$$
\end{itemize}
\end{remark}

\begin{proposition} \label{ciccic2} 
Let $G=C_d \rtimes C_m=\langle \beta \rangle \rtimes \langle \alpha \rangle$ where $C_d < H$, $d \mid (q+1)$, $d \ne 2$, $m \mid (q+1)$ of $\cMq$ such that $\alpha$ is of type (A); then the genus $\bar g$ of the quotient curve $\cH_q / G$ is given by one of the following values:
$$\bar g = \begin{cases}
(q^2-2q-3+4d)/(4d), & if \ d \ is \ odd, \ m=2, \ and \ \alpha \beta=\beta \alpha, \\
((q+1)(q-2-d)+4d)/(4d), &   if \ d \ is \ odd, \ m=2, \ and \ \alpha \beta \ne \beta \alpha, \\
((q+1)(q-1-m-2e)+2md)/(2md), & if \ d \ is \ odd, \ m>2, \ and \ \alpha\beta=\beta\alpha, \\
((q+1)(q-5)+4d) / (4d), & if \ d \ is \ even, \ m=2 \ and \ \alpha\beta=\beta\alpha,\\
((q+1)(q-3-d)+4d)/(4d), & \ if d \ is \ even, m=2 \ and \ \alpha\beta \ne \beta\alpha, \\
((q+1)(q-2-m-2e)+2md)/(2md), & if \ d \ is \ even, \ m>2, \ and \ \alpha\beta=\beta\alpha.
\end{cases}$$
where $e=\sum_{d^\prime \mid (m,d)} \varphi(d^\prime)$ and $\varphi(d^\prime)$ denotes the Euler Totient function of $d^\prime$.  
\end{proposition}

\begin{proof}
\textbf{Case 1. d is odd and $\bf m=2$} \ \\
Since $C_d < H$ and $d$ is odd, every $\sigma \in C_d$ is of type (B1) from Lemma \ref{pelem}. Let $\{P,Q,R\}$ be the fixed points of $\beta$. Assume that $\alpha$ and $\beta$ commute. This is equivalent to require that $\alpha$ is a homology of center $T$ where $T \in \{P,Q,R\}$. Since $d$ and $m$ are coprime, every element of type $\beta^\prime \alpha$ is of type (B1) because it fixed $P,Q$ and $R$ but it cannot fixed other points. From the Riemann-Hurwitz formula and Theorem \ref{caratteri} we have
$$(q^2-q-2)=2d(2 \bar g-2) +(d-1)0+(q+1)+(d-1)0.$$
Assume that $\alpha$ and $\beta$ do not commute. In this case we can assume that $\alpha(P)=Q$ and $\alpha(R)=R$. Up to conjugation we can assume that $P=(1:0:0)$, $Q=(0:1:0)$, $R=(0:0:1)$, where $\cH_q$ has equation \eqref{M1}. Thus, if $\beta^\prime \in C_d$ then $\beta^\prime$ admits a matrix representation $[\lambda,\lambda^{-1},1]$
for some $\lambda \in \mathbb{F}_{q^2}$ with $o(\lambda)=o(\beta^\prime)$, and hence
$$\alpha=\begin{pmatrix} 0 & 1 & 0 \\ 1 & 0 & 0 \\ 0 & 0 & 1 \end{pmatrix}.$$
By direct checking $o(\beta^\prime \alpha)=2$ and $G$ is a dihedral group. Now, from the Riemann-Hurwitz formula
$$(q^2-q-2)=2d(2 \bar g -2)+(d-1)0+(q+1)+(d-1)(q+1).$$

\textbf{Case 2. d is odd and $\bf m>2$} \ \\
Assume that $\alpha$ and $\beta$ commute. This is equivalent to require that $\alpha$ is a homology of center $T$ where $T \in \{P,Q,R\}$. Let $D=(d,m)$. As before we can assume up to conjugation that $P=(1:0:0)$, $Q=(0:1:0)$, $R=(0:0:1)$ and $\cH_q$ has equation \eqref{M1}. Thus $\beta=[\lambda,\lambda^{-1},1]$ and $\alpha$ is either $\alpha=[\gamma,\gamma,1]$, or $\alpha=[1,\gamma,1]$ or $\alpha=[1,1,\gamma]$; say $\alpha=[1,\gamma,1]$.
By direct checking, $\beta^h\alpha^k$ is of type (A) if and only if $\gamma^k\in\{\lambda^h,\lambda^{2h}\}$, otherwise $\beta^h\alpha^k$ is of type (B1).
This proves that if $D \ne 1$ then for every $d^\prime \mid D$, $G$ contains $2 \varphi(d^\prime)$ elements of type (A) other than the elements of $C_m$; as usual, set $\varphi(1)=0$. From the Riemann-Hurwitz formula
$$(q^2-q-2)=dm(2 \bar g -2)+(d-1)0+(m-1)(q+1)+2 \sum_{d^\prime \mid D} \varphi(d^\prime)(q+1).$$
We note that in Case 2 $\alpha$ and $\beta$ must commute, since $\alpha$ acts semiregularly outside its center and axis while $m>2$.
\ \\
\textbf{Case 3. d is even and $\bf m=2$} \ \\
In this case $C_d$ contains exactly one element of type (A) which has order $2$, and the other $d-2$ elements are of type (B1). Assume that $\alpha$ and $\beta$ commute. Arguing with the matrix representations as in Case 2, $G$ contains just a homology other than $\beta^{o(\beta)/2}$ and $\alpha$ which is given by $\beta^{o(\beta)/2} \alpha$ and the other elements are of type (B1). By the Riemann-Hurwitz formula
$$(q^2-q-2)=2d(2 \bar g-2)+ (d-2)0+(q+1)+(q+1)+(q+1)+(d-2)0.$$
Assume that $\alpha$ and $\beta$ do not commute. Using the matrix representation of Case 1, one can check that also in this case $G$ is a dihedral group. Thus,
$$(q^2-q-2)=2d(2 \bar g -2)+(d-2)0+(q+1)+(q+1)+(d-1)(q+1).$$
\ \\
\textbf{Case 4. d is even and $\bf m>2$} \ \\
Since $m>2$ and $\alpha$ acts semiregularly outside its center and axis, $\alpha$ and $\beta$ commute.
Arguing as in Case 2, the Riemann-Hurwitz formula yields
$$(q^2-q-2)=dm(2 \bar g -2)+(q+1)+(d-2)0+(m-1)(q+1)+2 \sum_{d^\prime \mid D} \varphi(d^\prime)(q+1).$$
\end{proof}

\begin{remark}\label{remopiuforte}
As in Remark \ref{remoforte}, Proposition \ref{ciccic2} provides the genus of $\cH_q/G$ for any $G$ as in the hypothesis of Proposition \ref{ciccic2}; conversely, a group $G$ such that the genus of $\cH_q/G$ is any of the genera stated in Proposition \ref{ciccic2} exists.
\end{remark}

\begin{proposition}\label{propsl25}
Let $G=SL(2,5) \rtimes C_m\leq\cMq$ where $p \geq 7$, $q \equiv 1 \pmod5$, $\langle \alpha \rangle=C_m$ and $m \mid (q+1)$ such that $\alpha$ is of type (A); then the genus $\bar g$ of the quotient curve $\cH_q / G$ is given by one of the following values 
$$\bar g= \begin{cases} \frac{(q+1)(q-1-2m)+4m}{240m}, \ if \ \alpha \in Z(\cMq), \ q \equiv_{12} 1, \ and \ m \ is \ odd, \\
\frac{(q+1)(q-1-2m)+64m}{240m}, \ if \ \alpha \in Z(\cMq), \ q \equiv_{12} 7, \ and \ m \ is \ odd, \\
\frac{(q+1)(q-2m-20D+19)+84m}{240m}, \ if \ \alpha \in Z(\cMq), \ q \equiv_{12} 5, \ and \ m \ is \ odd,  \\
\frac{(q+1)(q-2m-20D+19)+144m}{240m}, \ if \ \alpha \in Z(\cMq), \ q \equiv_{12} 11, \ and \ m \ is \ odd,\end{cases}$$
where $D=(m,3)$.
\end{proposition}

\begin{proof}
Table \ref{tabellasl25}, as a direct application of Theorem \ref{classificazione} summarizes the values of $i(\sigma)$, for $\sigma \in SL(2,5)$ according to the congruence of $q$ modulo $12$.

\begin{center}
\begin{table}
\begin{small}
\caption{Values of $i(\sigma)$ for $\sigma \in SL(2,5)$}\label{tabellasl25}
\begin{tabular}{|c|c||c|c|c|c|}
\hline Order of $\sigma$ & number of elements in $SL(2,5)$ & $q \equiv_{12} 1$ & $q \equiv_{12} 7$ & $q \equiv_{12} 5$ & $q \equiv_{12} 11$ \\
\hline 2 & 1 & q+1 & q+1 & q+1 & q+1 \\
\hline 3 & 20 & 2 & 2 & 0 & 0 \\
\hline 4 & 30 & 2 & 0 & 2 & 0 \\
\hline 5 & 24 & 2 & 2 & 2 & 2 \\
\hline 6 & 20 & 2 & 2 & 0 & 0 \\
\hline 10 & 24 & 2 & 2 & 2 & 2 \\
\hline  & & $i(\sigma)$ & $i(\sigma)$ & $i(\sigma)$ & $i(\sigma)$ \\
\hline
\end{tabular}
\end{small}
\end{table}
\end{center}

Assume that $\alpha$ and $SL(2,5)$ commute. Thus, $\alpha \in Z(\cMq)$. Since $|SL(2,5)|$ is even, then $m$ must be odd since otherwise $SL(2,5)$ and $C_m$ cannot be disjoint. Assume that $q \equiv_{12} 1$. From the Riemann-Hurwitz formula
$$q^2-q-2=120m(2 \bar g -2)+(q+1)+(m-1)(q+1) +(20 \cdot 2 + 24 \cdot 2 + 30 \cdot 2 + 20 \cdot 2 + 24 \cdot 2)$$
$$+(m-1)(q+1)+(m-1)(20 \cdot 2 + 24 \cdot 2 + 30 \cdot 2 + 20 \cdot 2 + 24 \cdot 2),$$
now the claim follows by direct computation. 
Assume that $q \equiv_{12} 7$. From the Riemann-Hurwitz formula
$$q^2-q-2=120m(2 \bar g -2)+(q+1)+(m-1)(q+1) +(20 \cdot 2 + 24 \cdot 2 + 30 \cdot 0 + 20 \cdot 2 + 24 \cdot 2)$$
$$+(m-1)(q+1)+(m-1)(20 \cdot 2 + 24 \cdot 2 + 30 \cdot 0 + 20 \cdot 2 + 24 \cdot 2).$$
Assume that $q \equiv_{12} 5$. Since $m$ is odd and we can get a homology of the form $\beta \alpha$, for $\beta \in SL(2,5)$ if and only if $o(\beta)=o(\alpha)$, we can construct homologies $\beta \alpha \in G$ if and only if $o(\alpha)=o(\beta)=3$. Denote $D=(m,3)$. From the Riemann-Hurwitz formula
$$q^2-q-2=120m(2 \bar g -2)+(q+1)+(m-1)(q+1) +(20 \cdot 0 + 24 \cdot 2 + 30 \cdot 2 + 20 \cdot 0 + 24 \cdot 2)$$
$$+(m-1)(q+1)+(m-1)(30 \cdot 2 + 24 \cdot 2 + 24 \cdot 2) + 20(D-1)(q+1).$$
Assume that $q \equiv_{12} 11$. Arguing as in the previous case, from the Riemann-Hurwitz formula
$$q^2-q-2=120m(2 \bar g -2)+(q+1)+(m-1)(q+1) +(20 \cdot 0 + 24 \cdot 2 + 30 \cdot 0 + 20 \cdot 0 + 24 \cdot 2)$$
$$+(m-1)(q+1)+(m-1)(24 \cdot 2 + 24 \cdot 2) + 20(D-1)(q+1).$$
We now assume that $\alpha \not\in Z(\cMq)$. Let $q \equiv_{12} 1$. From \cite[Proposition 1.2]{WALL}, if $m$ is odd then $G$ is a direct product of $SL(2,5)$ and $C_m$ and hence $\alpha \in Z(\cMq)$, a contradiction. We assume that $m$ is even. Since $2 \mid (q+1)$ but $4 \nmid (q+1)$, we can write $m=2\tilde m$, where $\tilde m$ is odd. If $\tilde m >1$ then $C_m = C_{\tilde m} \times C_2$ is a cyclic group generated by a homology $\alpha$, which has two proper subgroups $C_{\tilde m}$ and $C_2$ of homologies with different axes; a contradiction. Hence $m=2$ and $G=SL(2,5) \rtimes C_2$. By direct checking, with MAGMA, there are just 4 group of order 240 containing $SL(2,5)$, namely $G \cong SmallGroup(240,i)$ with $i \in \{89,90,93,94\}$. 
\begin{itemize}
\item $i=89$: cannot occur, since in this case $G$ must contain a unique involution; a contradiction.
\item $i=93$: then $G$ contains a cyclic normal subgroup $C_4=\langle \beta \rangle$ of order $4$. From Theorem \ref{classificazione}, $\beta$ is of type (B2) fixing three $\mathbb{F}_{q^2}$-rational points $P,Q,R$ such that $P,Q \in \cH_q(\mathbb{F}_{q^2})$ and $R \not\in \cH_q$. Every $\sigma \in G$ normalizes $\beta$ and hence either fixes $P$ and $Q$ or $\sigma(P)=Q$. Since every element of order $3$ must commute with $\beta$, because it cannot have an orbit of length $2$, we have a contradiction.
\item $i=94$: then $G=SL(2,5) \times C_2$; a contradiction. 
\item $i=90$: from the Riemann-Hurwitz formula,
$$q^2-q-2=480(2 \bar g -2)+21(q+1)+2(218);$$
since by direct computation $\bar g \not\in \mathbb{Z}$, this case is impossible.
\end{itemize}
Assume that $q \equiv_{12} 7$. Arguing as in the previous case $G=SL(2,5) \rtimes C_{2^h}$ where $h \geq 1$ and $2^h \mid (q+1)$. Then $\alpha$ acts on a set $I$ of $\mathbb{F}_{q^2}$-rational triangles which are the fixed points of the subgroups of order $4$ of $SL(2,5)$, which have not common fixed points as they do not commute. Since $|I|=15$ is odd, $\alpha$ must have at least a fixed triangle $T=\{P,Q,R\}$. Since $\alpha \not\in Z(\cMq)$ then $\alpha(P)=Q$ and hence there is a $\alpha$-orbit of length $2$. This implies that $o(\alpha)=2$ and $G=SL(2,5) \rtimes C_2$. As before $G \cong SmallGroup(240,i)$ for $i \in \{89,90,93,94\}$; arguing as above, we get a contradiction in each case.
%\begin{itemize}
%\item $i=89$: cannot occur, since $G$ must contain a unique involution; a contradiction.
%\item $i=90$:  from the Riemann-Hurwitz formula,
%$$q^2-q-2=480(2 \bar g -2)+21(q+1)+2(188),$$
%since by direct computation $\bar g \not\in \mathbb{Z}$, this case is impossible.
%\item $i=93$:  then $G$ contains a cyclic normal subgroup $C_4=\langle \beta \rangle$ of order $4$. From Theorem \ref{classificazione}, $\beta$ is of type (B1) fixing three $\mathbb{F}_{q^2}$-rational points $P,Q,R$ such that $P,Q,R \in \cH_q(\mathbb{F}_{q^2})$. Every $\sigma \in G$ normalizes $\beta$ and hence either fixes $P$ and $Q$ or $\sigma(P)=Q$. Since every element of order $3$ must commute with $\beta$, because it cannot have an orbit of length $2$, we have a contradiction.
%\item $i=94$: $G \cong SL(2,5) \times C_2$, a contradiction.
%\end{itemize}
The proofs for $q \equiv_{12} 5$ and $q \equiv_{12} 11$ are similar to the previous ones.
\end{proof}

\begin{remark}\label{remoforteforte}
As in Remark \ref{remoforte}, Proposition \ref{propsl25} provides the genus of $\cH_q/G$ for any $G$ as in the hypothesis of Proposition \ref{propsl25}; conversely, a group $G$ such that the genus of $\cH_q/G$ is any of the genera stated in Proposition \ref{ciccic2} exists.
\end{remark}

\begin{proposition}\label{sl23omologia}
Let $G=SL(2,3) \rtimes C_m$  of $\cMq$ where $p \geq 5$, $\langle \alpha \rangle=C_m$ and $m \mid (q+1)$ such that $\alpha$ is of type (A); then the genus $\bar g$ of the quotient curve $\cH_q / G$ is given by one of the following values:
$$\bar g= \begin{cases}
\frac{(q+1)(q-1-2m)+4m}{48m}, \ if \ \alpha \in Z(\cMq), \ q \equiv_{12} 1, \ and \ m \ is \ odd,\\
\frac{(q+1)(q-1-2m)+16m}{48m}, \ if \ \alpha \in Z(\cMq), \ q \equiv_{12} 7, \ and \ m \ is \ odd, \\
\frac{(q+1)(q-2m-8D+7)+36m}{48m}, \ if \ \alpha \in Z(\cMq), \ q \equiv_{12} 5, \ and \ m \ is \ odd,  \\
\frac{(q+1)(q-2m-8D+7)+48m}{48m}, \ if \ \alpha \in Z(\cMq), \ q \equiv_{12} 11, \ and \ m \ is \ odd, \\
\frac{(q^2-q-2-13(q+1)-68)}{96}+1, \ if \ \alpha \not\in Z(\cMq), \ q \equiv_{12} 1, \ and \ m=2, \\ \frac{(q^2-q-2-13(q+1)-36)}{96}+1, \ if \ \alpha \not\in Z(\cMq), \ q \equiv_{12} 5, \ and \ m=2,\\ \frac{(q^2-q-2-13(q+1)-32)}{96}+1, \ if \ \alpha \not\in Z(\cMq), \ q \equiv_{24} 7,  \ m=2, \ and \ G \cong SmallGroup(48,29),
\\
\frac{(q^2-q-2-13(q+1))}{96}+1, \ if \ \alpha \not\in Z(\cMq), \ q \equiv_{24} 23, \ m=2, \ and \ G \cong SmallGroup(48,29),\\
\frac{(q^2-q-2-10(q+1)-64)}{96}+1, \ if \ \alpha \not\in Z(\cMq), \ q \equiv_{12} 7,  \ m=2, \ and \ G \cong SmallGroup(48,33),
\\
\frac{(q^2-q-2-10(q+1))}{96}+1, \ if \ \alpha \not\in Z(\cMq), \ q \equiv_{12} 11, \ m=2, \ and \ G \cong SmallGroup(48,33),\\
\frac{(q^2-q-2-13(q+1)-128)}{192}+1, \ if \ \alpha \not\in Z(\cMq), \ q \equiv_{12} 7, \ m=4, \ and \ G \cong SmallGroup(96,74),\\
\frac{(q^2-q-2-13(q+1))}{192}+1, \ if \ \alpha \not\in Z(\cMq), \ q \equiv_{12} 11, \ m=4, \ and \ G \cong SmallGroup(96,74),\\
\frac{(q^2-q-2-21(q+1)-32-24S)}{192}+1, \ if \ \alpha \not\in Z(\cMq), \ q \equiv_{12} 7, \ m=4, \ and \ G \cong SmallGroup(96,67),\\
\frac{(q^2-q-2-21(q+1)-24S)}{192}+1, \ if \ \alpha \not\in Z(\cMq), \ q \equiv_{12} 11, \ m=4, \ and \ G \cong SmallGroup(96,67),
 \end{cases}$$
where $D=(m,3)$ and $S=0$ if $8 \mid (q+1)$, while $S=2$ if $8 \nmid (q+1)$.
\end{proposition}
\begin{proof}

Table \ref{tabellasl23}, as a direct application of Theorem \ref{classificazione} summarizes the values of $i(\sigma)$, for $\sigma \in SL(2,3)$ according to the congruence of $q$ modulo $12$.
\begin{center}
\begin{table}
\begin{small}
\caption{Values of $i(\sigma)$ for $\sigma \in SL(2,3)$}\label{tabellasl23}
\begin{tabular}{|c|c||c|c|c|c|}
\hline Order of $\sigma$ & number of elements in $SL(2,3)$ & $q \equiv_{12} 1$ & $q \equiv_{12} 7$ & $q \equiv_{12} 5$ & $q \equiv_{12} 11$ \\
\hline 2 & 1 & q+1 & q+1 & q+1 & q+1 \\
\hline 3 & 8 & 2 & 2 & 0 & 0 \\
\hline 4 & 6 & 2 & 0 & 2 & 0 \\
\hline 6 & 8 & 2 & 2 & 0 & 0 \\
\hline  & & $i(\sigma)$ & $i(\sigma)$ & $i(\sigma)$ & $i(\sigma)$ \\
\hline
\end{tabular}
\end{small}
\end{table}
\end{center}

Assume that $\alpha$ and $SL(2,3)$ commute. Thus, $\alpha \in Z(\cMq)$. Since $|SL(2,3)|$ is even, then $m$ must be odd since otherwise $SL(2,3)$ and $C_m$ cannot be disjoint. Assume that $q \equiv_{12} 1$. From the Riemann-Hurwitz formula
$$q^2-q-2=24m(2 \bar g -2)+(q+1)+(m-1)(q+1) +(8 \cdot 2 + 6 \cdot 2 + 8 \cdot 2)$$
$$+(m-1)(q+1)+(m-1)(8 \cdot 2 + 6 \cdot 2 + 8 \cdot 2),$$
now the claim follows by direct computation. 
Assume that $q \equiv_{12} 7$. From the Riemann-Hurwitz formula
$$q^2-q-2=24m(2 \bar g -2)+(q+1)+(m-1)(q+1) +(8 \cdot 2 + 6 \cdot 0 + 8 \cdot 2)$$
$$+(m-1)(q+1)+(m-1)(8 \cdot 2 + 6 \cdot 0 + 8 \cdot 2).$$
Assume that $q \equiv_{12} 5$. Since $m$ is odd and we can get a homology of the form $\beta \alpha$, for $\beta \in SL(2,5)$ if and only if $o(\beta)=o(\alpha)$, we can construct homologies $\beta \alpha \in G$ if and only if $o(\alpha)=o(\beta)=3$. Denote $D=(m,3)$. From the Riemann-Hurwitz formula
$$q^2-q-2=24m(2 \bar g -2)+(q+1)+(m-1)(q+1) +(8 \cdot 0 + 6 \cdot 2 + 8 \cdot 0)$$
$$+(m-1)(q+1)+(m-1)(6 \cdot 2) + 8(D-1)(q+1).$$
Assume that $q \equiv_{12} 11$. Arguing as in the previous case, from the Riemann-Hurwitz formula
$$q^2-q-2=24m(2 \bar g -2)+(q+1)+(m-1)(q+1) +(8 \cdot 0 + 6 \cdot 0 + 8 \cdot 0)$$
$$+(m-1)(q+1)+(m-1)(8 \cdot 0 + 6 \cdot 0+8 \cdot 0) + 8(D-1)(q+1).$$
We now assume that $\alpha \not\in Z(\cMq)$. From \cite[Proposition 1.2]{WALL}, if $m$ is odd then $G$ is a direct product of $SL(2,3)$ and $C_m$ and hence $\alpha \in Z(\cMq)$, a contradiction. We assume that $m$ is even. If either $q \equiv_{12} 1$ or $q \equiv_{12} 5$ then since $2 \mid (q+1)$ but $4 \nmid (q+1)$, we can write $m=2\tilde m$, where $\tilde m$ is odd. If $\tilde m >1$ then $C_m = C_{\tilde m} \times C_2$ is a cyclic group generated by a homology $\alpha$, which has two proper subgroups $C_{\tilde m}$ and $C_2$ of homologies with different axis; a contradiction. Hence $m=2$ and $G=SL(2,3) \rtimes C_2$. By direct checking, with MAGMA, there are just 4 group of order 48 containing $SL(2,3)$, namely $G \cong SmallGroup(48,i)$ for some $i \in \{28,29,32,33\}$. 
\begin{itemize}
\item $i=28$: this case cannot occur since $G$ has more than $1$ involution.
\item $i=29$: from the Riemann-Hurwitz formula
$$q^2-q-2=48(2 \bar g -2)+13(q+1)+2(12+8+8+6)$$
if $q \equiv_{12} 1$, and
$$q^2-q-2=48(2 \bar g -2)+13(q+1)+2(6+12)$$
if $q \equiv_{12} 5$. 
\item $i=32$: this case cannot occur as $G \not\cong SL(2,3) \times C_2$.
\item $i=33$: then $G$ contains a cyclic normal subgroup $C_4=\langle \beta \rangle$ of order $4$. From Theorem \ref{classificazione}, $\beta$ is of type (B2) fixing three $\mathbb{F}_{q^2}$-rational points $P,Q,R$ such that $P,Q \in \cH_q(\mathbb{F}_{q^2})$ and $R \not\in \cH_q$. Every $\sigma \in G$ normalizes $\beta$ and hence either fixes $P$ and $Q$ or $\sigma(P)=Q$. Since every element of order $3$ must commute with $\beta$, because it cannot have an orbit of length $2$, we have a contradiction. 
\end{itemize}
Assume that either $q \equiv_{12} 7$ or $q \equiv_{12}11$. Arguing as in the previous case $G=SL(2,3) \rtimes C_{2^h}$ where $h \geq 1$ and $2^h \mid (q+1)$. Assume that there exists $\alpha^\prime \in C_{2^h}$ such that $4 \mid o(\alpha^\prime)$. Then $\alpha^\prime$ acts on a set of $6$ triangles of type (B1), which are the fixed points of the elements of order $4$ in $SL(2,3)$, without fixed triangles. Since this set is made of $12$ points and a homology acts semiregularly outside its center and axis, we get that $o(\alpha^\prime)=4$. Thus either $G=SL(2,3) \rtimes C_2$ or $G=SL(2,3) \rtimes C_4$.
%\textcolor{red}{In effetti per $q=11$ magma mi dice che un $SL(2,3) \rtimes C_4$ esiste!}
We consider the case in which $G \cong SL(2,3) \rtimes C_2$. By direct checking, with MAGMA, there are just 4 group of order 48 containing $SL(2,3)$, namely $G \cong SmallGroup(48,i)$ for some $i \in \{28,29,32,33\}$. 
\begin{itemize}
\item $i=28$: this case cannot occur since $G$ has more than $1$ involution.
\item $i=29$: since there exists $\sigma \in G$ of order $8$ such that $\sigma^2$ is of type (B1), then a necessary condition is $8 \mid (q+1)$.
%\textcolor{red}{Non riesco a fare test in magma, $q$ diventa troppo grande, ma il genere viene intero.}
Let $\sigma \in G$ such that $o(\sigma)=8$. Then $\sigma$ can be either of type (A) or of type (B1). Since $\sigma^2$ is of type (B1) then $\sigma$ cannot be of type (A). From the Riemann-Hurwitz formula
$$q^2-q-2=48(2 \bar g -2)+13(q+1) +8 \cdot 2+ 8 \cdot 2 + 6 \cdot 0+ 12 \cdot 0$$
for $q \equiv_{12} 7$, and
$$q^2-q-2=48(2 \bar g -2)+13(q+1) + 8 \cdot 0+ 8 \cdot 0 + 6 \cdot 0+ 12 \cdot 0$$
for $q \equiv_{12} 11$. The claim now follows by direct computation.
\item $i=32$: this case cannot occur as $G \not\cong SL(2,3) \times C_2$.
\item $i=33$:
%\textcolor{red}{questo gruppo esiste sia per $q=11$ che per $q=19$!}
As before, when $q \equiv_{12} 11$, the elements of order $12$ are of type (B1) as they are all conjugated in $G$ and powers of some of them are of type (B1). The group $G$ contains a cyclic group of order $4$ which is not contained in $SL(2,3)$ and it is not conjugated to the other cyclic subgroups of order $4$. A direct computation shows that its generators are of type (A), otherwise  from the Riemann-Hurwitz formula $\bar g \not\in \mathbb{Z}$. From the Riemann-Hurwitz formula
$$q^2-q-2=48(2 \bar g -2)+7(q+1)+8 \cdot 2 + 8 \cdot 2 +16 \cdot 2+ 8 \cdot 0-2(q+1),$$
if $q \equiv_{12} 7$ and
$$q^2-q-2=48(2 \bar g -2)+7(q+1) +8 \cdot 0 + 8 \cdot 0 +16 \cdot 0+ 8 \cdot 0-2(q+1),$$
if $q \equiv_{12} 11$. The claim now follows by direct computation.
\end{itemize}
We consider the case in which $G \cong SL(2,3) \rtimes C_4$. By direct checking with MAGMA, there are 4 groups of order 96 containing $SL(2,3)$ as a normal subgroup with $G/SL(2,3) \cong C_4$, namely $G \cong SmallGroup(96,i)$ for some $i \in \{66, 67, 69, 74\}$. 
\begin{itemize}
\item $i=66$ or $i=69$: in this case $G$ has three normal subgroups of order $2$.
Therefore $G$ fixes pointwise a self-polar triangle and is abelian, a contradiction.
\item $i=67$: since $Z(G) \cong C_4$, looking at the conjugacy classes of $G$, from the Riemann-Hurwitz formula we have
$$q^2-q-2=96(2 \bar g -2)+7(q+1)+8 \cdot 2 +2(q+1)+6 \cdot 0+12 \cdot 0 +12 \cdot (q+1) + 8 \cdot 2$$
$$+24 \cdot \begin{cases} 0, \ if \ 8 \mid (q+1), \\ 2, \ otherwise \end{cases}+16 \cdot 2$$
if $q \equiv_{12} 7$, and
$$q^2-q-2=96(2 \bar g -2)+7(q+1)+8 \cdot 0+2(q+1)+6 \cdot 0+12 \cdot 0 +12 \cdot (q+1) + 8 \cdot 0$$
$$+24 \cdot \begin{cases} 0, \ if \ 8 \mid (q+1), \\ 2, \ otherwise \end{cases}+16 \cdot 0$$
if $q \equiv_{12} 11$. The claim now follows by direct checking.
\item $i=74$: since there exists $\sigma \in G$ of order $8$ such that $\sigma^2$ is of type (B1), then a necessary condition is $8 \mid (q+1)$. Moreover, $G$ contains 16 elements of order $8$: 12 of them are of type (B1) and the remaining 4 are of type (A). In fact $Z(G) \cong C_8$, and so $G$ contains at least $4$ homologies of order $8$, but since the remaining elements of order $8$  $\sigma_1,\ldots,\sigma_{12}$ are such that $\sigma_i^2 \in SL(2,3)$ is of type (B1) for every $i=1,\ldots,12$ then they cannot be of type (A).
%\textcolor{red}{Non riesco a fare test in magma, $q$ diventa troppo grande, ma il genere viene intero.}
The elements of order $4$ split into 2 conjugacy classes: 2 elements are of type (A) and belong to $Z(G)$, while the remaining are of type (B1). From the Riemann-Hurwitz formula
$$q^2-q-2=96(2 \bar g -2)+7(q+1)+8 \cdot 2 + 2(q+1)+6 \cdot 0 + 8 \cdot 2 + 4(q+1)+12 \cdot 0 + 16 \cdot 2 + 32 \cdot 2$$
if $q \equiv_{12} 7$, and 
$$q^2-q-2=96(2 \bar g -2) + 7(q+1)+  8 \cdot 0 + 2(q+1)+ 6 \cdot 0+8 \cdot 0 + 4(q+1)+12 \cdot 0+ 16 \cdot 0 +32 \cdot 0$$
if $q \equiv_{12} 11$. The claim now follows by direct computation.
\end{itemize}
\end{proof}

\begin{remark}
If $\alpha\in Z(\cMq)$, then the genera $\bar g$ given in Proposition \ref{sl23omologia} are actually obtained for some $G\leq PGU(3,q)$.
If $\alpha\notin Z(\cMq)$, then $G$ as in the hypothesis of Proposition \ref{sl23omologia} may exist or not. For instance, tests with MAGMA show that $PGU(3,11)$ has subgroups $SmallGroup(48,33)$ and $SmallGroup(96,67)$, $PGU(3,13)$ has a subgroup $SmallGroup(48,29)$, and $PGU(3,19)$ has a subgroup $SmallGroup(48,33)$.
\end{remark}

\begin{proposition}\label{Komologia}
If there exists a subgroup $G=K \rtimes C_m$ of $\cMq$ where $K \cong SmallGroup(48,28)$, $p \geq 5$, $q^2 \equiv 1 \mod16$, $\langle \alpha \rangle=C_m$ and $m \mid (q+1)$ such that $\alpha$ is of type (A); then the genus $\bar g$ of the quotient curve $\cH_q / G$ is given by one of the following values 
$$\bar g=\begin{cases} \frac{(q+1)(q-1-2m)+4m}{96m}, \ if \ \alpha \in Z(\cMq), \ m \ is \ odd, \ and \ q \equiv_{24} 1,13, \\
\frac{(q+1)(q+7-2m-8D)+36m}{96m}, \ if \ \alpha \in Z(\cMq), \ m \ is \ odd, \ and \ q \equiv_{24} 5,17, \\
\frac{(q+1)(q-1-2m)+64m}{96m}, \ if \ \alpha \in Z(\cMq), \ m \ is \ odd, \ and \ q \equiv_{24} 7,19, \\
\frac{(q+1)(q+7-2m-8D)+96m}{96m}, \ if \ \alpha \in Z(\cMq), \ m \ is \ odd, \ and \ q \equiv_{24} 11,23, \\
\end{cases}$$
where $D=(m,3)$.
\end{proposition}

\begin{proof}
Table \ref{tabellasg4828} summarizes the values of $i(\sigma)$, for $\sigma \in K$ according to the congruence of $q$ modulo $24$.
\begin{center}
\begin{table}
\begin{small}
\caption{Values of $i(\sigma)$ for $\sigma \in K$}\label{tabellasg4828}
\begin{tabular}{|c|c||c||c||c||c||}
\hline $o(\sigma)$ & $\#\{\gamma \mid o(\gamma)=o(\sigma)\}$ & $q \equiv_{24} 1,13$ & $q \equiv_{24} 5,17$ & $q \equiv_{24} 7,19$ & $q \equiv_{24} 11,23$ \\
\hline 2 & 1 & q+1 & q+1 & q+1 & q+1 \\
\hline 3 & 8 & 2 & 0 & 2 & 0 \\
\hline 4 & 18 & 2 & 2 & 0 & 0 \\
\hline 6 & 8 & 2 & 0 & 2 & 0 \\
\hline 8 & 12 & 2 & 2 & 0 & 0 \\
\hline
\end{tabular}
\end{small}
\end{table}
\end{center}
Assume that $\alpha$ and $K$ commute, that is, $\alpha \in Z(\cMq)$. Since $|K|$ is even, then $m$ must be odd since otherwise $K$ and $C_m$ cannot  be disjoint. Assume that $q \equiv_{24} 1,13$. Then from the Riemann-Hurwitz formula
$$q^2-q-2=48m(2 \bar g -2)+(q+1)+(m-1)(q+1)+(m-1)(q+1)$$
$$+2(8+18+8+12)+2(m-1)(8+18+8+12).$$
Assume that $q \equiv_{24} 5,17$. As before we observe that we can have $\beta \alpha \in G$ of type (A) if and only if $o(\beta)=o(\alpha)=3$. Then from the Riemann-Hurwitz formula
$$q^2-q-2=48m(2 \bar g -2)+(q+1)+(m-1)(q+1)+(m-1)(q+1)$$
$$+(8 \cdot 0 + 18 \cdot 2+ 8 \cdot 0 + 12 \cdot 2)+(m-1)(18 \cdot 2 + 12 \cdot 2)+8(D-1)(q+1),$$
where $D=(3,m)$.
Assume that $q \equiv_{24} 7,19$. Then from the Riemann-Hurwitz formula
$$q^2-q-2=48m(2 \bar g -2)+(q+1)+(m-1)(q+1)+(m-1)(q+1)$$
$$+2(8+8)+0(18+12)+2(m-1)(8+8).$$
Assume that $q \equiv_{24} 11,23$. As before we observe that we can have $\beta \alpha \in G$ of type (A) if and only if $o(\beta)=o(\alpha)=3$. Then from the Riemann-Hurwitz formula
$$q^2-q-2=48m(2 \bar g -2)+(q+1)+(m-1)(q+1)+(m-1)(q+1)$$
$$+(8 \cdot 0 + 18 \cdot 0+ 8 \cdot 0 + 12 \cdot 0)+(m-1)(18 \cdot 0 + 12 \cdot 0)+8(D-1)(q+1),$$
where as before $D=(3,m)$.

We now assume that $\alpha \not\in Z(\cMq)$. We note that $K$ has a subgroup $H \cong SL(2,3)$ which is characteristic, as it is the unique subgroup of $K$ having order equal to $24$.  Thus, $H$ is normal in $G$ and $\langle H, \alpha \rangle = H \rtimes C_m \cong SL(2,3) \rtimes C_m$ because $C_m$ and $H$ are disjoint. If $m$ is odd then $\alpha$ commutes with $H$ from \cite[Proposition 1.2]{WALL}, and hence $\alpha \in Z(\cMq)$, a contradiction. Thus we can assume that $m$ is even. We note that if $m=2^h \tilde m$, where $\tilde m$ is odd, then from \cite[Proposition 1.2]{WALL} $C_{\tilde m}<C_m$ is a group of elements of type (A) belonging to $Z(\cMq)$, a contradiction to $\alpha \not\in Z(\cMq)$. In particular this shows that $G=K \rtimes C_{2^h}$ with $h \geq 1$. If $q \equiv_{24} 1,13,5,17$ then $2 \mid (q+1)$ but $4 \nmid (q+1)$, thus $G=K \rtimes C_2$. Suppose that $q \equiv_{24} 7,19$ and that $o(\alpha)>4$. Then $\alpha^\prime$ acts on a set of $4$ triangle of type (B2) which are the sets of the fixed points of the elements of order $3$ contained in $K$. Since $\alpha^\prime$ cannot fix any of these points, we have a contradiction.  Thus, if $q\equiv_{24} 7,19$, then either $G=K \rtimes C_2$ or $G=K \rtimes C_4$. Assume that $q \equiv_{24} 11,23$. Arguing as before, since $\alpha$ acts on a set of $4$ triangles of type (B1) which are the fixed points of the elements of order $3$, we get that either $m=2$ or $m=4$.

By direct checking with MAGMA, there are no groups of order $96$ of the form $G=K \rtimes C_2$, then both the cases $m=2$ and $m=4$ are impossible.
\end{proof}

\begin{remark}
Note that groups $G$ with $\bar g$ as in the thesis of Proposition \ref{Komologia} do exist for any $q$; conversely, the genera of $\cH_q/G$ with $G$ satisfying the hypothesis of Proposition \ref{Komologia} are all given in this proposition.
\end{remark}

\begin{proposition} \label{qotto}
Let $G=Q_8\rtimes C_m\leq PGU(3,q)$, where $Q_8$ is the quaternion group of order $8$ and $C_m=\langle\alpha\rangle$ with $m\mid(q+1)$, $\alpha$ of type (A). Then $\cH_q/G$ has genus
$$\bar g =\begin{cases} \frac{(q+1)(q-1-2m)+4m}{16m}, \ if \ \alpha \in Z(\cMq), \ q \equiv_4 1, \ m \ is \ odd, \\ \frac{(q+1)(q-1-2m)+16m}{16m}, \ if \ \alpha \in Z(\cMq), \ q \equiv_4 3, \ m \ is \ odd, \\ \frac{(q-1)(q-5)}{32}, \ if \ \alpha \notin Z(\cMq), \ q \equiv_4 1, \ m=2, \ and  \ G \cong SmallGroup(16,8),\\ \frac{q^2-6q+25}{32}, \ if \ \alpha \notin Z(\cMq), \ q \equiv_8 7, \ m=2, \ and  \ G \cong SmallGroup(16,8), \\ \frac{(q-1)(q-7)}{32}, \ if \ \alpha \notin Z(\cMq), \ q \equiv_8 1, \ m=2, \ and  \ G \cong SmallGroup(16,13), \\ \frac{(q-7)(q-3)}{32}, \ if \ \alpha \notin Z(\cMq), \ q \equiv_4 3, \ m=2, \ and  \ G \cong SmallGroup(16,13),\\ \frac{(q-5)^2}{48}, \ if \ \alpha \notin Z(\cMq), \ q \equiv_4 1, \ m=3, \ and  \ G \cong SmallGroup(24,3),\\ \frac{q^2-10q+37}{48}, \ if \ \alpha \notin Z(\cMq), \ q \equiv_4 3, \ m=3, \ and  \ G \cong SmallGroup(24,3),\end{cases}$$
\end{proposition}

\begin{proof}
Suppose $\alpha \in Z(\cMq)$, so that $G=Q_8 \times C_m$ and $m$ is odd. If $4$ divides $q-1$ then from the Riemann-Hurwitz formula,
$$(q+1)(q-2)=8m(2 \bar g -2) +6m \cdot 2 + (2m-1)(q+1);$$
if $4$ divides $q+1$ then
$$(q+1)(q-2)=8m(2 \bar g -2) + (2m-1)(q+1).$$
Suppose $\alpha \not\in Z(\cMq)$. Assume that $m=2$. By direct checking in MAGMA the only possibilities are $G \cong SmallGroup(16,8)$ or $G \cong SmallGroup(16,13)$. Let $G \cong SmallGroup(16,8)$. Then 
$$(q+1)(q-2)=16(2 \bar g -2) +5(q+1)+6 \cdot 2 + 4 \cdot 2$$
if $4$ divides $q-1$, and
$$(q+1)(q-2)=16(2 \bar g -2)+5(q+1)$$
if $8$ divides $q+1$; the case $4 \mid (q+1)$ and $8 \nmid (q+1)$ is not possibile by direct inspection of $SmallGroup(16,8)$.
Let $G \cong SmallGroup(16,13)$. Then
$$(q+1)(q-2)=16(2 \bar g -2)+ 7(q+1) + 8 \cdot 2$$
if $8$ divides $q-1$ (the case $4 \mid (q-1)$ and $8 \nmid (q-1)$ as $\bar g$ is integer),
$$(q+1)(q-2)=16(2 \bar g -2)+7(q+1)+2(q+1)$$
if $4$ divides $q+1$. From now on, $m >2$. Assume that $\alpha$ fixes pointwise the fixed points of an element $\beta \in Q_8$ of order $4$. Then $m=4$. This case is not possible; in fact, by direct inspection with MAGMA, all groups of type $Q_8 \rtimes C_4$ have $3$ normal involutions. Such a groups in $PGU(3,q)$ are abelian, a contradiction.
The remaining cases require that $\alpha$ acts with long orbits on a set of $6$ points and hence either $m=3$ or $m=6$. 
Suppose $m=3$. There exists just one group $SmallGroup(24,3)$ of type $Q_8 \rtimes C_3$, and 
$$(q+1)(q-2)=24(2 \bar g -2) +1(q+1)+8(q+1)+6 \cdot 2$$
if $4$ divides $q-1$, and 
$$(q+1)(q-2)=24(2 \bar g -2) +1(q+1)+8(q+1)$$
if $4$ divides $q+1$. Suppose $m=6$. By direct checking with MAGMA a group of type $Q_8 \rtimes C_6$ does not exist.
\end{proof}

\begin{remark}
If $\alpha \in Z(\cMq)$ then $\bar g$ as in Proposition \ref{qotto} exists; other instances are the following: $PGU(3,11)$ contains $SmallGroup(16,13)$ and $SmallGroup(24,3)$, $PGU(3,13)$ contains $SmallGroup(16,8)$, $PGU(3,17)$ contains $SmallGroup(24,3)$.
\end{remark}

\begin{proposition}\label{Dicomologia}
If there exists a subgroup $G=Dic(n) \rtimes C_m$ of $\cMq$, $n>2$, where $Dic(n)$ is the dicyclic group $Dic(n)=\langle a,x \mid a^{2n}=1, \ x^2=a^n, \ x^{-1} a x =a^{-1} \rangle$ of order $4n$ for $n \mid \frac{q \pm 1}{2}$, $\langle \alpha \rangle=C_m$ and $m \mid (q+1)$ such that $\alpha$ is of type (A); then the genus $\bar g$ of the quotient curve $\cH_q / G$ is given by one of the following values.
\begin{itemize}
\item If $q \equiv_{4} 1$:
$$\bar g=\begin{cases} \frac{(q+1)(q-1-2m)+4m}{8mn}, \ if \ n \mid (q-1)/2, \ m \ is \ odd, \ \alpha \in Z(\cMq),\\
\frac{(q+1)(q-1-2m-2D)+4mn}{8mn}, \ if \ n \mid (q+1)/2, \ m \ is \ odd, \ \alpha \in Z(\cMq),\\
\frac{()(q+1)(q-3-2n)+4n}{16n}, \ if \ n \mid (q-1)/2, \ \alpha \not\in Z(\cMq), \ and \ m=2,\\
\frac{(q+1)(q-5-2n)+12n}{16n}, \ if \ n\mid(q+1)/2, \ \alpha \not\in Z(\cMq), \ and \ m=2,
\end{cases}$$
\item If $q \equiv_{4} 3$:
$$\bar g= \begin{cases} \frac{(q+1)(q-1-2m)+4m(n+1)}{8mn}, \ if \ n \mid (q-1)/2, \ m,n \ are \ odd, \ \alpha \in Z(\cMq),\\
\frac{(q+1)(q-1-2m-2D)+8mn}{8mn}, \ if \ n \mid (q+1)/2, \ m \ is \ odd, \ \alpha \in Z(\cMq),  \end{cases}$$
where $D=\sum_{d \mid (m,n)}\varphi(d)$, and $\varphi$ denote the Euler totient function.
\end{itemize}
\end{proposition}

\begin{proof}
We observe that if $n \ne 2$, then $C_{2n}=\langle a \rangle$ is a characteristic subgroup of $Dic(n)$ as it is the unique cyclic group of $Dic(n)$ of order $2n$. In fact the following is the complete list of the elements of $Dic(n)$,
\begin{itemize}
\item its unique central element, which is the involution $a^n=x^2$,
\item the nontrivial elements $ \sigma \in \langle a \rangle \setminus \{a^n \}$. They are $2n-2$, and the conjugacy class of each of them has length 2 (contains just $\sigma$ and $\sigma^{-1}$). Thus, there are $n-1$ conjugacy classes of elements of type $\sigma$.
\item the elements $\beta \in Dic(n) \setminus \langle a \rangle$ are divided into two conjugacy classes: 
\begin{enumerate}
\item the conjugacy class of $x$, which contains the elements of order $4$ of type $a^{2k}x$ for $k \geq 0$,
\item the conjugacy class of $ax$, containing the remain elements of type $a^{2k+1}x$ for $k \geq 0$. Also these elements are of order $4$, as $(ax)^2=x(x^{-1}ax)ax=xa^{-1}ax=x^2$.
\end{enumerate}
\end{itemize}
This observation proves that $Dic(n)$ has one element of order $2$, $2n-2$ nontrivial elements of order which divides $2n$ and is greater than $2$, and $2n$ elements of order $4$. Since $\alpha$ acts by conjugation on the set of the cyclic subgroups of $Dic(n)$ of order $4$, and an element of order $4$ is either of type (B1) or (B2) from Theorem \ref{classificazione}, we have that $\alpha$ acts on a set of $n+1$ triangles if $n$ is even, while it acts on a set of $n$ triangles if $n$ is odd. Assume that
$\alpha \in Z(\cMq)$. Since $2 \mid |Dic(n)|$, $m$ must be odd.
\\ \\
\textbf{Case 1: $ \bf q \equiv_{4} 1$,  $ \bf n \mid (q-1)/2$}
\\ All the nontrivial elements of $Dic(n) \setminus \{a^n\}$  are of type (B2), while $a^n$ is of type (A). From the Riemann-Hurwitz formula
$$q^2-q-2=4nm(2 \bar g -2)+(q+1)+2(4n-2)+(m-1)(q+1)$$
$$+(m-1)(q+1)+2(m-1)(4n-2).$$
\\ 
\textbf{Case 2: $ \bf q \equiv_{4} 1$,  $ \bf n \mid (q+1)/2$}
\\ Here all the nontrivial elements of $C_{2n} \setminus \{a^n\}$ are of type (B1), $a^n$ is of type (A) and the remaining $2n$ are of type (B2). From the Riemann-Hurwitz formula
$$q^2-q-2= 4nm(2 \bar g -2)+(q+1)+(2n-2)0+(2n)2+(m-1)(q+1)$$
$$+(m-1)(q+1)+2\sum_{d \mid (m,n)} \varphi(d)(q+1)+(m-1)(2n)2$$
\\
\textbf{Case 3: $ \bf q \equiv_{4} 3$,  $ \bf n \mid (q-1)/2$} \\
Since $a$ is of type (B2), $n$ must be odd. In fact if $n$ is even, then $C_{2n}$ contains elements of order $4$, and then of type (B1), which are powers of an element of type (B2), a contradiction. Thus, the nontrivial elements of $C_{2n} \setminus \{a^n\}$ are of type (B2), the $2n$ elements of order $4$ are of type (B1) and $a^n$ is of type (A). From the Riemann-Hurwitz formula
$$q^2-q-2=4mn(2 \bar g -2)+(q+1)+(2n-2)2+(2n)0+(m-1)(q+1)$$
$$+(m-1)(q+1)+2(m-1)(2n-2).$$
\\
\textbf{Case 4: $ \bf q \equiv_{4} 3$,  $ \bf n \mid (q+1)/2$} \\
In this case the elements of $Dic(n) \setminus \{a^n\}$ are all of type (B1), while $a^n$ is of type (A). From the Riemann-Hurwitz formula
$$q^2-q-2= 4nm(2 \bar g -2)+(q+1)+(4n-2)0+(m-1)(q+1)$$
$$+(m-1)(q+1)+2\sum_{d \mid (m,n)} \varphi(d)(q+1).$$

We now assume that $\alpha \not\in Z(\cMq)$. 
\\ \\
\textbf{Case 1: $ \bf q \equiv_{4} 1$,  $ \bf n \mid (q-1)/2$}

Since $C_{2n}$ is normal in $G$, and its generator $a$ is of type (B2) we have that $o(\alpha)=2$, since homologies have long orbits outside its center and axis and $\alpha \not\in Z(\cMq)$. Furthermore by direct checking in matrix representation, as in the proof of Proposition 4.5, every element $a^k \alpha$ where $k \ne n$ is of order $2$. Looking at $\beta \alpha$ for $\beta \not\in C_{2n}$ we observe that $\beta$ and $\alpha$ has the same orbit of length $2$ given by the fixed points of $\alpha$ which is of type (B2). Thus, these two points are fixed by $\beta \alpha$. Since $\beta \alpha$ cannot be a homology, as the fixed points of $\beta$ are not fixed by $\beta \alpha$ we get that $\beta \alpha$ is of type (B2) as well. From the Riemann-Hurwitz formula
$$q^2-q-2=8n(2 \bar g -2)+(q+1)+2(4n-2)+(q+1) + (2n-1)(q+1)+(2n)2.$$
\\
\textbf{Case 2: $ \bf q \equiv_{4} 1$,  $ \bf n \mid (q+1)/2$} \\
Since $C_{2n}$ is normal in $G$, and its generator $a$ is of type (B1), we have that either $o(\alpha)=2$ or $\alpha$ and $C_{2n}$ commute, since homologies have long orbits apart from their center and axis and $\alpha \not\in Z(\cMq)$.

Assume that $o(\alpha)=2$ and that $\alpha$ and $C_{2n}$ do not commute.
Let $\mathcal H_q$ have equation \eqref{M1}. Up to conjugation, $a$ is the diagonal matrix $[\lambda,\lambda^i,1]$, with $\lambda$ a $2n$-th primitive root of unity, and
$$\alpha=\begin{pmatrix} 0 & 1 & 0 \\ 1 & 0 & 0 \\ 0 & 0 & 1 \end{pmatrix}, \quad x=\begin{pmatrix} 0 & A & 0 \\ -A^{-1} & 0 & 0 \\ 0 & 0 & 1 \end{pmatrix},$$
with $A^{q+1}=1$.
By direct checking, $x^{-1} a x =a^{-1}$ if and only if $i=-1$.
Moreover, $\alpha x \alpha \in Dic_n$ if and only if $ord(A)\mid 2n$.
As $(a^k \alpha)^2=1$, every element of type $a^k \alpha$ is of type (A).
Also, $a^k x\alpha =[\lambda^k A, -(\lambda^k A)^{-1},1]$ is of type (A) if $\lambda^k=\pm A^{-1}$, while $a^k x\alpha$ is of type (B1) otherwise.
From the Riemann-Hurwitz formula,
$$q^2-q-2=8n(2\bar g-2)+(q+1)+(2n)2+2n(q+1)+2(q+1).$$

Assume that $\alpha$ and $C_{2n}$ commute.
As before, $\mathcal H_q$ has equation \eqref{M1}, $a=\diag[\lambda,\lambda^{-1},1]$. Also, $\alpha=\diag[1,\mu,1]$. Since $\alpha$ acts on the elements of order $4$ in $Dic_n\setminus C_{2n}$, we have $ord(\mu)\mid 2n$. Thus $\mu=\lambda^j$ for some $j$; hence, $a^j\alpha^2=\diag[\mu,\mu,1]\in Z(\mathcal M_q)$ and $G=Dic_n\times\langle a^j\alpha^2\rangle$. This case has already been considered.
\\
\textbf{Case 3: $ \bf q \equiv_{4} 3$,  $ \bf n \mid (q-1)/2$} \\
Suppose that $\alpha$ and $C_{2n}$ do not commute. Let $\mathcal H_q$ have equation \eqref{M2}; up to conjugation, $a=[\lambda^2,\lambda,1]$ with $\lambda$ a $2n$-th primitive root of unity,
$$x=\begin{pmatrix} 0 & 0 & A \\ 0 & 1 & 0 \\ -A^{-1} & 0 & 0 \end{pmatrix}, \quad \alpha=\begin{pmatrix} 0 & 0 & 1 \\ -0 & 1 & 0 \\ 1 & 0 & 0 \end{pmatrix},$$
where $A^{q-1}=-1$.
By direct checking, $\alpha x\alpha\in Dic_n$ if and only if $\lambda^j=-A^{-2}$ for some $j$, which implies $A^{2n}=1$. Since $2n\mid(q-1)$, this is a contradiction to $A^{q-1}=-1$.
Thus, $\alpha$ and $C_{2n}$ commute. This implies that $\alpha\in Z(\mathcal M_q)$, and this case has already been considered.
\end{proof}

\begin{remark}
From the proof of the previous proposition, we note that all the integers $\bar g$ given in Proposition \ref{Dicomologia} actually occur as genera of some quotient $\cH_q/G$; viceversa, if $G$ satisfies the hypothesis, then the genus of $\cH_q/G$ is given in the thesis of Proposition \ref{Dicomologia}.
\end{remark}

\section{Further results}

\begin{proposition}
Let $p\geq2$ and $G\leq PGU(3,q)$ with $G\cong PGU(3,\bar q)$, where $q=p^h$, $\bar q=p^k$, $k$ divides $h$, and $h/k$ is odd.
Then the quotient curve $\cH_q/G$ has genus
$$\bar g= 1+\frac{q^2-q-2-\Delta}{2\bq^3(\bq^3+1)(\bq^2-1)},$$
where
$$ \Delta= (\bar q -1)(\bar q^3+1)\cdot(q+2) + (\bar q^3-\bar q)(\bar q^3+1)\cdot2 + \bar q(\bq^4-\bq^3+\bq^2)\cdot(q+1)$$
$$ + (\bq^2-\bq-2)\frac{(\bq^3+1)\bq^3}{2}\cdot2 + (\bq-1)\bq(\bq^3+1)\bq^2\cdot1 + (\bq^2-\bq)\frac{\bq^6+\bq^5-\bq^4-\bq^3}{3}\cdot\delta, $$
with
$$ \delta= \begin{cases} 3, & \textrm{if} \quad(\bq^2-\bq+1)\mid(q^2-q+1), \\ 0, & \textrm{if} \quad(\bq^2-\bq+1)\mid(q+1)\;\textrm{ and }\;\bq\ne2.\\ \end{cases} $$
\end{proposition}

\begin{proof}
Up to conjugation $G$ is the group of all $\sigma \in PGU(3,q)$ such that $\sigma$ is defined over $\mathbb{F}_{\bq^2}$.
First, we classify the elements of $PGU(3,\bar q)$ seen as the automorphism group of a Hermitian curve $\cH_\bq$, using the order statistics of $PGU(3,\bq)$ and Lemma \ref{classificazione}.
\begin{enumerate}
\item There are exactly $(\bar q -1)(\bar q^3+1)$ elements of type (C).
In fact, for each $P\in\cH_\bq(\mathbb F_{\bq^2})$, there exist exactly $\bq-1$ elations in $PGU(3,\bq)$ with center $P$.
%, and the remaining $\bq^3-\bq$ nontrivial elements of the Sylow $p$-subgroup of $PGU(3,\bq)$ fixing $P$ are of type (D).
\item There are exactly $(\bar q^3-\bar q)(\bar q^3+1)$ elements of type (D). In fact, they are the $p$-elements of $PGU(3,\bq)$ which are not of type (C).
\item There are exactly $\bar q(\bq^4-\bq^3+\bq^2)$ elements of type (A). In fact, for each $P\in PG(2,\bq^2)\setminus\cH_\bq$ there exist exactly $\bq$ homologies in $PGU(3,\bq)$ with center $P$.
\item There are exactly $(\bq^2-\bq-2)\frac{(\bq^3+1)\bq^3}{2}$ elements of type (B2). In fact, for each couple $\{P,Q\}\subset\cH_\bq(\mathbb F_{\bq^2})$ there exist exactly $\bq^2-2$ nontrivial elements of $PGU(3,\bq)$ fixing $P$ and $Q$, and $\bq$ of them are homologies with axis $PQ$.
\item There are exactly $(\bq^2-\bq)\frac{\bq^6+\bq^5-\bq^4-\bq^3}{3}$ elements of type (B3). In fact, any point $P\in\cH_{\bq^2}(\mathbb{F}_{\bq^6})\setminus\cH_{\bq}(\mathbb{F}_{\bq^2})$ determines a unique triangle $\{P,\Phi_{\bq^2}(P),\Phi_{\bq^2}^2(P)\}\subset\cH_{\bq}(\mathbb{F}_{\bq^6})$ fixed by a Singer subgroup of order $\bq^2-\bq+1$; see the proof of Lemma \ref{classificazione} given in \cite{MZRS}.
\item There are exactly $(\bq-1)\bq(\bq^3+1)\bq^2$ elements of type (E). In fact, consider a couple $\{P,Q\}$ with $P\in\cH_\bq(\mathbb F_{\bq^2})$ and $Q\in PG(2,\bq^2)\cap \ell_P$, where $\ell_P$ is the tangent line to $\cH_\bq$ at $P$. Any element of type (E) fixing $P$ and $Q$ is uniquely obtained as the product of an elation of center $P$ and a homology of center $Q$; thus there exist exaclty $(\bq-1)\bq$ such elements.
\item The remaining $\bq^4(\bq-1)^2(\bq^2-\bq+1)/6$ nontrivial elements are of type (B1).
\end{enumerate}

Now we describe the elements in each class (1) - (7) according to their geometry with respect to $\cH_q$, using Lemma \ref{classificazione}.
\begin{itemize}
\item[(i)] The elements in class (1) are of type (C).
In fact, let $\bar S$ be one of the $\bq^3+1$ Sylow $p$-subgroups of $PGU(3,\bq)$ and $S$ be the Sylow $p$-subgroup of $PGU(3,q)$ containing $\bar S$.
Note that $S$ is a trivial intersection set, since $\cH_q$ has zero $p$-rank; see \cite[Theorem 11.133]{HKT}.
Consider the explicit representation of $S$ given in \cite[Section 3]{GSX}, where $\cH_q$ has norm-trace equation and $S$ fixes the point at infinity.
By direct computation, there are $\bq$ elements of $\bar S$ in the center of $S$. Thus, $\bar S$ has exactly $\bq-1$ elements of type (C); see \cite[Equation (2.12)]{GSX}.
\item[(ii)] The elements in class (2) are of type (D).
In fact, with $\bar S$ as in Case (i), the claim follows from Case (i) counting the remaining nontrivial elements of $\bar S$.
\item[(iii)] Let $\sigma\in PGU(3,\bq)$ be in class (3) or (4).
Then $\sigma$ is contained in the pointwise stabilizer $D\cong C_{\bq+1}\times C_{\bq+1}$ of a self-polar triangle with respect to $\cH_\bq$; see Theorem \ref{Mit}.
Let $C_{q+1}\times C_{q+1}\leq PGU(3,q)$ be the poitwise stabilizer of a self-polar triangle $T$ with respect to $\cH_q$, such that $D\leq C_{q+1}\times C_{q+1}$.
Up to conjugation, $\cH_q$ has Fermat equation \eqref{M1} and $T$ is the fundamental triangle, so that
$$ D=\{(X,Y,T)\mapsto(\lambda X,\mu Y,T)\mid \lambda^{\bq+1},\mu^{\bq+1}=1\}. $$
By direct computation, $D$ contains $3\bq$ elements of type (A) and $\bq^2-\bq$ elements of type (B1).
Since $PGU(3,\bq)$ is transitive on $PG(2,\bq^2)\setminus\cH_q$, there are exactly $\frac{|PGU(3,\bq)|}{6|D|}$ self-polar triangles $T^\prime$ with respect to $\cH_\bq$, whose pointwise stabilizer $D^\prime$ is conjugated to $D$ under $PGU(3,\bq)$.
Note that $D$ and $D^\prime$ intersect non-trivially if and only if $T$ and $T^\prime$ have a vertex $P$ in common; in this case, $(D\cap D^\prime)\setminus\{id\}$ is made by $\bq$ homologies with center $P$.
The number of points in $PG(2,\bq^2)\setminus\cH_q$ lying on the polar of $P$ is $\bq^2-\bq$; hence, the number of $D^\prime$ which intersect $D$ non-trivially is $(\bq^2-\bq)/2$.
Therefore, by direct computation, the number of element of type (A) in $PGU(3,\bq)$ is exactly the number $\bq(\bq^4-\bq^3+\bq^2)$ of elements in class (3), and the remaining $\bq^4(\bq-1)^2(\bq^2-\bq+1)/6$ elements of the subgroups of $PGU(3,\bq)$ conjugated to $D$ are of type (B1) and in class (7).
\item[(iv)] Let $\sigma\in PGU(3,\bq)$ be in class (4).
Since the order $o(\sigma)$ of $\sigma$ divides $\bq^2-1$ but not $\bq+1$, we have that $o(\sigma)$ divides $q^2-1$ but not $q+1$, as $q$ is an odd power of $\bq$.
Therefore $\sigma$ is of type (B2).
\item[(v)] Let $\sigma\in PGU(3,\bq)$ be in class (6).
Since the order of $\sigma$ is $p\cdot d$ where $d>1$ and $p\nmid d$, $\sigma$ is of type (E).
\item[(vi)] Let $\sigma\in PGU(3,\bq)$ be in class (5).
By direct checking, the order $o(\sigma)$ of $\sigma$ divides either $q^2-q+1$ or $q+1$.
Assume that $\bq^2-\bq+1 \ne 3$. If $\bq^2-\bq+1 \mid (q^2-q+1)$ then every $\sigma$ in class (5) is of type (B3), by Lemma \ref{classificazione}. If $\bq^2-\bq+1 \mid (q+1)$ then $\sigma$ is either of type (A) or (B1). We note that $\sigma$ is contained in the maximal subgroup (iv) of $PGU(3,\bq)$ in Theorem \ref{Mit}, which is a semidirect product $C_{\bq^2-\bq+1} \rtimes C_3$, where $\sigma \in C_{\bq^2-\bq+1}$ and $C_3$ does not commute with any subgroup of $C_{\bq^2-\bq+1}$. Therefore $\sigma$ cannot be of type (A), because otherwise every element which normalizes $\sigma$ should commute with $\sigma$.
Then $\sigma$ of type (B1).
Now assume $\bq^2-\bq+1=3$, that is $\bq=2$.
Let $S$ be a Singer subgroup of $PGU(3,q)$ containing $\sigma$ as in case (iv) of Theorem \ref{Mit}. Using the explicit description of $S$ given in \cite{CKT1,CKT2}, it is easily seen that $S\cap PGU(3,\bq)$ is cyclic of order $3$.
Then $\sigma$ is of type (B3).
\end{itemize}
From Theorem \ref{caratteri}, the degree $\Delta$ of the different divisor of the cover $\cH_\bq\to\cH_\bq/G$ is
%$$ \Delta= (\bar q -1)(\bar q^3+1)\cdot(q+2) + (\bar q^3-\bar q)(\bar q^3+1)\cdot2 + \bar q(\bq^4-\bq^3+\bq^2)\cdot(q+1)$$
%$$ + (\bq^2-\bq-2)\frac{(\bq^3+1)\bq^3}{2}\cdot2 + (\bq^2-\bq)\frac{\bq^6+\bq^5-\bq^4-\bq^3}{3}\cdot3 + (\bq-1)\bq(\bq^3+1)\bq^2\cdot1. $$
$$ \Delta= (\bar q -1)(\bar q^3+1)\cdot(q+2) + (\bar q^3-\bar q)(\bar q^3+1)\cdot2 + \bar q(\bq^4-\bq^3+\bq^2)\cdot(q+1)$$
$$ + (\bq^2-\bq-2)\frac{(\bq^3+1)\bq^3}{2}\cdot2 + (\bq-1)\bq(\bq^3+1)\bq^2\cdot1 + (\bq^2-\bq)\frac{\bq^6+\bq^5-\bq^4-\bq^3}{3}\cdot\delta, $$
where $\delta=3$ if $(\bq^2-\bq+1)\mid(q^2-q+1)$, and $\delta=0$ if $(\bq^2-\bq+1)\mid(q+1)$ and $\bq\ne2$.
From the Riemann-Hurwitz genus formula, the claim follows.
\end{proof}

%(bq -1)*(bq^3+1)*(q+2) + (bq^3-bq)*(bq^3+1)*2 + bq*(bq^4-bq^3+bq^2)*(q+1) + (bq^2-bq-2)*(bq^3+1)*bq^3 + (bq^2-bq)*(bq^6+bq^5-bq^4-bq^3) + (bq-1)*bq*(bq^3+1)*bq^2

\begin{proposition}\label{Esiste1}
Let $p\geq2$, $2\ne n\mid(q+1)$, $n$ prime, $3\mid(q-1)$, $3\mid(n-1)$.
%Suppose there exists $i\in\mathbb N$ such that $n\mid(i^2-i+1)$; this happens for instance when $n\equiv1\pmod6$.
Then there exists a subgroup $G=C_n \rtimes C_3 \leq PGU(3,q)$ such that the genus $\bar g$ of the quotient curve $\cH_q / G$ is
%$$\bar g=\frac{(q+1)(q-k_1-k_2)+2n}{2nm}.$$
$$\bar g=\frac{(q+1)(q-2)+2n}{6n}.$$
\end{proposition}

\begin{proof}
Let $\cH_q: X^{q+1}+Y^{q+1}+Z^{q+1}=0$, and consider the following automorphisms of $\cH_q$
$$\alpha:(X,Y,Z) \mapsto (\lambda X, \lambda^{i}Y,Z),$$
where $\lambda^n=1$ and
$$\beta: (X,Y,Z) \mapsto (AY,BZ,X),$$
where $A^{q+1}=B^{q+1}=1$.
Since $n$ is prime and $n\equiv1\pmod3$, we have $\gamma:=\beta^{-1}\alpha\beta=\alpha^{n-i}$. Thus, $G:=\langle\alpha,\beta\rangle$ is a nontrivial semidirect product $C_n\rtimes C_3$.
Also, the elements of $C_n$ are of type (B1), while the elements of $G\setminus C_n$ have order $3$ and are of type (B2).
From the Riemann-Hurwitz formula,
$$ (q+1)(q-2)=3n(2\bar g-2)+2n\cdot2. $$
\end{proof}

\begin{proposition}\label{Esiste2}
Let $p>3$. Then there exists a subgroup $G\leq PGU(3,q)$ such that $G$ is isomorphic to the alternating group $A_4$ and the genus $\bar g$ of $\cH_q/G$ is
$$ \bar g =\begin{cases} 1+\frac{(q+1)(q-2)}{24}, & \textrm{if}\quad 3\mid(q+1); \\ \frac{(q+1)(q-2)+8}{24}, & \textrm{if}\quad 3\mid(q-1). \end{cases}$$
\end{proposition}

\begin{proof}
Define $G=\langle\alpha,\beta,\gamma\rangle$ with $\alpha:(X,Y,Z)\mapsto(-X,Y,Z)$, $\beta:(X,Y,Z)\mapsto(X,-Y,Z)$, $\alpha:(X,Y,Z)\mapsto(Y,Z,X)$. Then $G\cong A_4$, and from the Riemann-Hurwitz formula the claim follows.
\end{proof}

\begin{proposition}\label{Esiste3}
Let $p=3$. Then there exists a subgroup $G\leq PGU(3,q)$ such that $G$ is isomorphic to the symmetric group $S_3$ and the genus $\bar g$ of $\cH_q/G$ is
$$ \bar g =\begin{cases} \frac{(q+1)(q-5)+6}{12}, & \textrm{if}\quad 3\mid(q+1); \\ \frac{(q+1)(q-5)+2}{12}, & \textrm{if}\quad 3\mid(q-1). \end{cases}$$
\end{proposition}

\begin{proof}
Let $\cH_q$ have equation \eqref{M1}. Let $M\cong (C_{q+1}\times C_{q+1})\rtimes K$, with $K\cong S_3$, be the maximal subgroup of $PGU(3,q)$ fixing the fundamental triangle $T=\{P_1,P_2,P_3\}$.
Assume $3\mid(q+1)$.
Since $M$ contains a unique subgroup $C_3$ generated by an element $\alpha$ of type (B1), we have $G=\langle\alpha,\beta\rangle\cong S_3$, where $\beta$ is any involution of $M$.
Then from the Riemann-Hurwitz formula
$$ (q+1)(q-2)=6(2\bar g-2)+3(q+2). $$
Assume $3\mid(q-1)$ and take $G=K$.
Then from the Riemann-Hurwitz formula
$$ (q+1)(q-2)=6(2\bar g-2)+3(q+2)+2\cdot2. $$
\end{proof}

We provide a model for a quotient curve $\cH_q/ \langle \sigma \rangle$ where $\sigma$ is of type (E).

\begin{proposition}\label{EqTipoE}
For $q=p^h$ let $\sigma \in PGU(3,q)$ be of type (E) and of order $pd$. Choose $\lambda \in \overline{\mathbb{F}}_q$ such that $\lambda^h=-1$.
Then a plane model for the quotient curve $\cH_q / \langle \sigma \rangle$ is given by
\begin{equation} \label{tipoE}
y^{(q+1)/d}=\sum_{i=1}^{h} \lambda^{i-1} x^{q/p^i}.
\end{equation}
\end{proposition}
\begin{proof}
Let $\overline{\mathbb{F}}_q(\cH_q)=\overline{\mathbb{F}}_q(x,y)$ with $y^{q+1}=x^q+x$. Up to conjugation $\sigma: (x,y) \mapsto (x+c, ay)$ where $a$ is a $d$-th primitive root of unity and $c^p=\lambda c$. Define $\xi=y^d$ and $\eta=x^p-\lambda x$. Clearly $\overline{\mathbb{F}}_q(\xi,\eta) \subseteq \overline{\mathbb{F}}_q(\cH_q)^{\langle \sigma \rangle}$, where $\overline{\mathbb{F}}_q(\cH_q)^{\langle \sigma \rangle}$ denotes the subfield of $\overline{\mathbb{F}}_q(\cH_q)$ which is fixed by $\langle \sigma \rangle$. Since $[\overline{\mathbb{F}}_q(y,\eta):\overline{\mathbb{F}}_q(\xi,\eta)]=d$ and $[\overline{\mathbb{F}}_q(\xi,x): \overline{\mathbb{F}}_q(\xi,\eta)]=p$, the compositum $\overline{\mathbb{F}}_q(x,y)$ has degree $dp$ over $\overline{\mathbb{F}}_q(\xi,\eta)$ by Abhyankar's Lemma \cite[Theorem 3.9.1]{Sti} and hence $\overline{\mathbb{F}}_q(\cH_q)^{\langle \sigma \rangle}=\overline{\mathbb{F}}_q(\xi,\eta)$. Consider the conventional $p$-associate polynomials $x-\lambda$ and $x^h-\lambda^h$ of $x^p-\lambda x$ and $x^{q}+x$; see \cite[Definition 3.58]{LN}. Since $(x^h-\lambda^h)/(x-\lambda)=\sum_{i=1}^{h} \lambda^{i-1} x^{h-i}$, we have by \cite[Lemma 3.59]{LN} that 
$$\xi^{(q+1)/d}=x^q+x=\sum_{i=1}^{h} \lambda^{i-1} \eta^{q/p^i}.$$
The equation $\xi^{(q+1)/d}=\sum_{i=1}^{h} \lambda^{i-1} \eta^{q/p^i}$ is absolutely irreducible as it defines the Kummer extension $\overline{\mathbb{F}}_q(\xi,\eta): \overline{\mathbb{F}}_q(\eta)$ of degree $(q+1)/d$.
\end{proof}

\begin{remark}
Let $\mathcal X$ be the quotient curve $\cH_q/\langle\sigma\rangle$ for some nontrivial $\sigma\in PGU(3,q)$, $q=p^h$.
Then one of the following cases occurs:
\begin{enumerate}
\item $o(\sigma)=p$, $g(\mathcal X)=p^{h-1}(q-p)/2$, and $$\mathcal X:\quad \sum_{i=1}^h y^{q/{p^i}}+\omega x^{q+1}=0,$$ with $\omega^{q-1}=-1$; see \cite[Theorem 2.1 (II) (1)]{CKT2}.
\item $p\geq3$, $o(\sigma)=p$, $g(\mathcal X)=p^{h-1}(q-1)/2$, and $$\mathcal X:\quad y^q+y=\bigg(\sum_{i=1}^h x^{q/{p^i}}\bigg)^2;$$
see \cite[Theorem 2.1 (II) (2)]{CKT2}.
\item $o(\sigma)=pd$ with $1<d\mid(q+1)$, $g(\mathcal X)=\left(\frac{q+1}{d}-1\right)\left(p^{h-1}-1\right)/2$, and $$\mathcal X:\quad y^{(q+1)/d}=\sum_{i=1}^{h} \lambda^{i-1} x^{q/p^i};$$
see \cite[Theorem 4.4]{GSX} and Proposition \ref{EqTipoE}.
\item $o(\sigma)=d\mid(q^2-q+1)$, $g(\mathcal X)=\frac{1}{2}\left(\frac{q^2-q+1}{d}-1\right)$, and $$\mathcal X:\quad F(x,y,1)=0,$$
where $F(X,Y,Z)$ is defined in \cite[Remark 5.5]{CKT1}; see \cite[Theorem 5.1]{CKT1}.
\item $o(\sigma)=d\mid(q^2-1)$, $g(\mathcal X)=\frac{1}{2d}[q+1-\gcd(d,q+1)](q-1)$, and
$$ \mathcal X:\quad y^{(q^2-1)/d}=x(x+1)^{q-1}; $$
see \cite[Corollary 4.9 and Example 6.3]{GSX}.
\item $o(\sigma)=d\mid(q+1)$, $\sigma$ is of type (B1), and $g(\mathcal X)=1+\frac{1}{2d}(q+1)(q+1-r_1-r_2-r_3)$ for some divisors $r_1,r_2,r_3$ of $q+1$; see \cite[Theorem 5.8]{GSX}.
%, and $$ \mathcal X:\quad y^{(q^2-1)/d}=x(x+1)^{q-1}; $$
\item $p=2$, $o(\sigma)=4$, $\sigma$ is of type (D), and $g(\mathcal X)=(q^2-2q)/8$; see \cite[Theorem 3.3]{GSX}.
\end{enumerate}
Partial results on models for $\mathcal X$ in Case (6) are given in \cite[Example 6.4]{GSX} and \cite[Equation (3.1)]{GHKT}; yet, a model of $\mathcal X$ is still unknown for a general $\sigma$ of type (B1).
In case (7), a model for $\mathcal X$ is unknown.
\end{remark}

\section{Spectrum of genera of quotients of $\cH_q$, $q \leq 29$}

In this section the complete spectrum $\Sigma_q$ of genera of $\mathbb{F}_{q^2}$-maximal curves which are Galois subcovers of the Hermitian curve $\cH_q$ is determined for all $q \leq 29$.

The proof relies on the results of \cite{MZRS}.  A case-by-case analysis of all integers $g$ with $1<g\leq g(\cH_{q})$ is combined with the classical bounds
\begin{equation}\label{limitiamo} \frac{|\cH_{q}(\F_{q^2})|}{|\cH_{q}/G(\F_{q^2})|} \leq|G|\leq \frac{2g(\cH_{q})-2}{2g(\cH_{q}/G)-2} \,, \end{equation}
This leads us to
look inside the structure of the groups $G$ satisfying \eqref{limitiamo} and compute the genus of $\cH_{q}/G$, for $g>1$.
For each $g>1$ in $\Sigma_{q}$, the following tables provide a classification of the groups $G$ for which $\cH_{q}/G$ has genus $g$.

\begin{center}
\begin{table}[htbp]
\begin{small}
\caption{Quotient curves $\cH_{2}/G$, \ $G \leq \PGU(3,2)$}\label{tabella1}
\begin{tabular}{|c|c|c|}
\hline \bf $g$ & $|G|$ & \textrm{structure of $G$} \\
\hline\hline 1 & 1 & \textrm{trivial group.}\\
\hline 0 & 2 & $G=C_2=\langle\sigma\rangle$, $\sigma$ of type (C).\\
\hline
\end{tabular}
\end{small}
\end{table}
\end{center}
\begin{center}
\begin{table}[htbp]
\begin{small}
\caption{Quotient curves $\cH_{3}/G$, \ $G \leq \PGU(3,3)$}\label{tabella2}
\begin{tabular}{|c|c|c|}
\hline \bf $g$ & $|G|$ & \textrm{structure of $G$} \\
\hline\hline 3 & 1 & \textrm{trivial group.}\\
\hline 1 & 2 & $G=C_2=\langle\sigma\rangle$, $\sigma$ of type (A).\\
\hline 0 & 3 & $G=C_3=\langle\sigma\rangle$, $\sigma$ of type (C).\\
\hline
\end{tabular}
\end{small}
\end{table}
\end{center}

\begin{center}
\begin{table}[htbp]
\begin{small}
\caption{Quotient curves $\cH_{4}/G$, \ $G \leq \PGU(3,4)$}\label{tabella3}
\begin{tabular}{|c|c|c|}
\hline \bf $g$ & $|G|$ & \textrm{structure of $G$} \\
\hline\hline 6 & 1 & \textrm{trivial group.}\\
\hline 2 & 3 & $G=C_3=\langle\sigma\rangle$, $\sigma$ of type (B2).\\
\hline 1 & 4 & $G=C_4=\langle\sigma\rangle$, $\sigma$ of type (D).\\
\hline 0 & 5 & $G=C_5=\langle\sigma\rangle$, $\sigma$ of type (A).\\
\hline
\end{tabular}
\end{small}
\end{table}
\end{center}

\begin{center}
\begin{table}[htbp]
\begin{small}
\caption{Quotient curves $\cH_{5}/G$, \ $G \leq \PGU(3,5)$}\label{tabella4}
\begin{tabular}{|c|c|c|}
\hline \bf $g$ & $|G|$ & \textrm{structure of $G$} \\
\hline\hline 10 & 1 & \textrm{trivial group.}\\
\hline 4 & 2 & $G=C_2=\langle\sigma\rangle$, $\sigma$ of type (A).\\
\hline 3 & 3 & $G=C_3=\langle\sigma\rangle$, $\sigma$ of type (B3).\\
\hline 2 & 4 & $G=C_4=\langle\sigma\rangle$, $\sigma$ of type (B2).\\
\hline 1 & 3 & $G=C_3=\langle\sigma\rangle$, $\sigma$ of type (A).\\
\hline 0 & 5 & $G=C_5=\langle\sigma\rangle$, $\sigma$ of type (C).\\
\hline
\end{tabular}
\end{small}
\end{table}
\end{center}

\begin{center}
\begin{table}[htbp]
\begin{small}
\caption{Quotient curves $\cH_{7}/G$, \ $G \leq \PGU(3,7)$}\label{tabella5}
\begin{tabular}{|c|c|c|}
\hline \bf $g$ & $|G|$ & \textrm{structure of $G$} \\
\hline\hline 21 & 1 & \textrm{trivial group.}\\
\hline 9 & 2 & $G=C_2=\langle\sigma\rangle$, $\sigma$ of type (A).\\
\hline 7 & 3 & $G=C_3=\langle\sigma\rangle$, $\sigma$ of type (B2).\\
\hline 5 & 4 & $G=C_4=\langle\sigma\rangle$, $\sigma$ of type (B1).\\
\hline 3 & 4 & $G=C_4=\langle\sigma\rangle$, $\sigma$ of type (A).\\
\hline 2 & 6 & $G=Sym(3)=\langle\alpha\rangle\rtimes\langle\beta\rangle$, $\alpha$ of type (B2), $\beta$ of type (A). \\
\hline 1 & 8 & $G=C_4 \times C_2=\langle\alpha \rangle \times \langle \beta \rangle$, $\alpha$ of type (A), $\beta$ of type (A).\\
\hline 0 & 7 & $G= C_7=\langle\sigma\rangle$, $\sigma$ of type (C). \\
\hline
\end{tabular}
\end{small}
\end{table}
\end{center}

%\textcolor{red}{Osservazione in $q=7$: I turchi mettono $g=14$ utilizzando male Thm 3 di Abdon-Quoos!!}

\begin{center}
\begin{table}[htbp]
\begin{small}
\caption{Quotient curves $\cH_{8}/G$, \ $G \leq \PGU(3,8)$}\label{tabella6}
\begin{tabular}{|c|c|c|}
\hline \bf $g$ & $|G|$ & \textrm{structure of $G$} \\
\hline\hline 28 & 1 & \textrm{trivial group.}\\
\hline 12 & 2 & $G=C_2=\langle\sigma\rangle$, $\sigma$ of type (C).\\
\hline 10 & 3 & $G=C_3=\langle\sigma\rangle$, $\sigma$ of type (B1).\\
\hline 9 & 3 & $G=C_3=\langle\sigma\rangle$, $\sigma$ of type (B3).\\
\hline 7 & 3 & $G=C_3=\langle\sigma\rangle$, $\sigma$ of type (A).\\
\hline 6 & 4 & $G=C_4=\langle\sigma\rangle$, $\sigma$ of type (D).\\
\hline 4 & 7 & $G=C_7=\langle\sigma\rangle$, $\sigma$ of type (B2). \\
\hline 3 & 6 & $G=C_6=\langle\sigma\rangle$, $\sigma$ of type (E).\\
\hline 2 & 8 & $G= C_4 \times C_2=\langle\alpha \rangle \times \langle \beta \rangle$, $\alpha$ of type (D), $\beta$ of type $(C)$. \\
\hline 1 & 19 & $G=C_19=\langle\sigma\rangle$, $\sigma$ of type (B3).\\
\hline 0 & 9 & $G=C_9=\langle\sigma\rangle$, $\sigma$ of type (A).\\
\hline
\end{tabular}
\end{small}
\end{table}
\end{center}

\begin{center}
\begin{table}[htbp]
\begin{small}
\caption{Quotient curves $\cH_{9}/G$, \ $G \leq \PGU(3,9)$}\label{tabella7}
\begin{tabular}{|c|c|c|}
\hline \bf $g$ & $|G|$ & \textrm{structure of $G$} \\
\hline\hline 36 & 1 & \textrm{trivial group.}\\
\hline 16 & 2 & $G=C_2=\langle\sigma\rangle$, $\sigma$ of type (A).\\
\hline 12 & 3 & $G=C_3=\langle\sigma\rangle$, $\sigma$ of type (D).\\
\hline 9 & 3 & $G=C_3=\langle\sigma\rangle$, $\sigma$ of type (C).\\
\hline 8 & 4 & $G=C_4=\langle\sigma\rangle$, $\sigma$ of type (B2).\\
\hline 6 & 4 & $G=C_2 \times C_2=\langle\alpha\rangle \times \langle \beta \rangle$, $\alpha$ of type (A), $\beta$ of type (A).\\
\hline 4 & 5 & $G=C_5=\langle\sigma\rangle$, $\sigma$ of type (A). \\
\hline 3 & 9 & $G=C_3 \times C_3=\langle\alpha\rangle \times \langle \beta \rangle$, $\alpha$ of type (C), $\beta$ of type (D).\\
\hline 2 & 8 & $G= D_8=C_4 \rtimes C_2=\langle\alpha \rangle \rtimes \langle \beta \rangle$, $\alpha$ of type (B2), $\beta$ of type $(A)$. \\
\hline 1 & 15 & $G=C_{15}=\langle\sigma\rangle$, $\sigma$ of type (E).\\
\hline 0 & 10 & $G=C_{10}=\langle\sigma\rangle$, $\sigma$ of type (A).\\
\hline
\end{tabular}
\end{small}
\end{table}
\end{center}

\begin{center}
\begin{table}[htbp]
\begin{small}
\caption{Quotient curves $\cH_{11}/G$, \ $G \leq \PGU(3,11)$}\label{tabella8}
\begin{tabular}{|c|c|c|}
\hline \bf $g$ & $|G|$ & \textrm{structure of $G$} \\
\hline\hline 55 & 1 & \textrm{trivial group.}\\
\hline 25 & 2 & $G=C_2=\langle\sigma\rangle$, $\sigma$ of type (A).\\
\hline 19 & 3 & $G=C_3=\langle\sigma\rangle$, $\sigma$ of type (B1).\\
\hline 18 & 3 & $G=C_3=\langle\sigma\rangle$, $\sigma$ of type (B3).\\
\hline 15 & 3 & $G=C_3=\langle\sigma\rangle$, $\sigma$ of type (A).\\
\hline 13 & 4 & $G=C_4=\langle\sigma\rangle$, $\sigma$ of type (B1).\\
\hline 11 & 5 & $G=C_5=\langle\sigma\rangle$, $\sigma$ of type (B2).\\
\hline 10 & 4 & $G=C_4=\langle\sigma\rangle$, $\sigma$ of type (A).\\
\hline 9 & 6 & $G=C_6=\langle\sigma\rangle$, $\sigma$ of type (B1).\\
\hline 7 & 8 & $G=Q_8$.\\
\hline 5 & 6 & $G=C_6=\langle\sigma\rangle$, $\sigma$ of type (A).\\
\hline 4 & 8 & $G=D_8=C_4 \rtimes C_2=\langle\alpha\rangle \rtimes \langle \beta \rangle$, $\alpha$ of type (B1), $\beta$ of type (A). \\
\hline 3 & 10 & $D_{10}=C_5 \rtimes C_2=\langle\alpha\rangle \rtimes \langle \beta \rangle$, $\alpha$ of type (B2), $\beta$ of type (A).\\
\hline 2 & 15 & $G=C_{15}=\langle\sigma\rangle$, $\sigma$ of type (B2). \\
\hline 1 & 37 & $G=C_{37}=\langle\sigma\rangle$, $\sigma$ of type (B3).\\
\hline 0 & 12 & $G=C_{12}=\langle\sigma\rangle$, $\sigma$ of type (A).\\
\hline
\end{tabular}
\end{small}
\end{table}
\end{center}

\begin{center}
\begin{table}[htbp]
\begin{small}
\caption{Quotient curves $\cH_{13}/G$, \ $G \leq \PGU(3,13)$}\label{tabella9}
\begin{tabular}{|c|c|c|}
\hline \bf $g$ & $|G|$ & \textrm{structure of $G$} \\
\hline\hline 78 & 1 & \textrm{trivial group.}\\
\hline 36 & 2 & $G=C_2=\langle\sigma\rangle$, $\sigma$ of type (A).\\
\hline 26 & 3 & $G=C_3=\langle\sigma\rangle$, $\sigma$ of type (B2).\\
\hline 18 & 4 & $G=C_4=\langle\sigma\rangle$, $\sigma$ of type (B2).\\
\hline 15 & 4 & $G=C_2 \times C_2=\langle\alpha\rangle \times \langle \beta \rangle$, $\alpha$ of type (A), $\beta$ of type (A).\\
\hline 12 & 6 & $G=C_6=\langle\sigma\rangle$, $\sigma$ of type (B2).\\
\hline 10 & 6 & $G=S_3$.\\
\hline 9 & 8 & $G=Q_8$.\\
\hline 6 & 8 & $G=D_8=C_4 \rtimes C_2=\langle\alpha\rangle \rtimes \langle \beta \rangle$, $\alpha$ of type (B2), $\beta$ of type (A).\\
\hline 5 & 12 & $G=A_4$.\\
\hline 4 & 21 & $G=C_7 \rtimes C_3=\langle\alpha\rangle \rtimes \langle \beta \rangle$, $\alpha$ of type (B1), $\beta$ of type (B2).\\
\hline 3 & 14 & $G=D_{14}=C_7 \rtimes C_2=\langle\alpha\rangle \rtimes \langle \beta \rangle$, $\alpha$ of type (B1), $\beta$ of type (A). \\
\hline 2 & 21 & $G=C_{21}=\langle\sigma\rangle$, $\sigma$ of type (B2). \\
\hline 0 & 14 & $G=C_{14}=\langle\sigma\rangle$, $\sigma$ of type (A).\\
\hline
\end{tabular}
\end{small}
\end{table}
\end{center}

\begin{center}
\begin{table}[htbp]
\begin{small}
\caption{Quotient curves $\cH_{16}/G$, \ $G \leq \PGU(3,16)$}\label{tabella10}
\begin{tabular}{|c|c|c|}
\hline \bf $g$ & $|G|$ & \textrm{structure of $G$} \\
\hline\hline 120 & 1 & \textrm{trivial group.}\\
\hline 56 & 2 & $G=C_2=\langle\sigma\rangle$, $\sigma$ of type (C).\\
\hline 40 & 3 & $G=C_3=\langle\sigma\rangle$, $\sigma$ of type (B2).\\
\hline 28 & 4 & $G=C_4=\langle\sigma\rangle$, $\sigma$ of type (D).\\
\hline 24 & 4 & $G=C_2 \times C_2=\langle\alpha\rangle \times \langle \beta \rangle$, $\alpha$ of type (A), $\beta$ of type (A).\\
\hline 16 & 6 & $G=Sym(3)=C_3 \rtimes C_2=\langle\alpha\rangle \rtimes \langle \beta \rangle$, $\alpha$ of type (B2) and $\beta$ of type (C).\\
\hline 12 & 8 & $G=C_4 \times C_2=\langle\alpha\rangle \times \langle \beta \rangle$, $\alpha$ of type (D) and $\beta$ of type (C).\\
\hline 8 & 8 & $G=C_2 \times C_2 \times C_2$.\\
\hline 6 & 16 & $G=C_4 \times C_4=\langle\alpha\rangle \times \langle \beta \rangle$, $\alpha$ of type (D), $\beta$ of type (D).\\
\hline 4 & 16 & $G=C_2 \times C_2 \times C_4$.\\
\hline 2 & 32 & $G=C_{2} \times C_4 \times C_4$. \\
\hline 1 & 64 & $G=C_4 \times C_4 \times C_4$. \\
\hline 0 & 17 & $G=C_{17}=\langle\sigma\rangle$, $\sigma$ of type (A).\\
\hline
\end{tabular}
\end{small}
\end{table}
\end{center}

%\textcolor{red}{ERRORE DEI TURCHI: mettono $g=60$ come se in Abdon-Quoos ci fosse un $C_2$ generato da un punto retta!!!Fanno lo stesso errore con $g=30,15,14,7$!!}

\begin{center}
\begin{table}[htbp]
\begin{small}
\caption{Quotient curves $\cH_{17}/G$, \ $G \leq \PGU(3,17)$}\label{tabella17}
\begin{tabular}{|c|c|c|}
\hline \bf $g$ & $|G|$ & \textrm{structure of $G$} \\
\hline\hline 136 & 1 & \textrm{trivial group.}\\
\hline 64 & 2 & $G=C_2=\langle\sigma\rangle$, $\sigma$ of type (A).\\
\hline 46 & 3 & $G=C_3=\langle\sigma\rangle$, $\sigma$ of type (B1).\\
\hline 45 & 3 & $G=C_3=\langle\sigma\rangle$, $\sigma$ of type (B3).\\
\hline 40 & 3 & $G=C_3=\langle\sigma\rangle$, $\sigma$ of type (A).\\
\hline 32 & 4 & $G=C_4=\langle\sigma\rangle$, $\sigma$ of type (B2).\\
\hline 28 & 4 & $G=C_2 \times C_2=\langle\alpha\rangle \times \langle \beta \rangle$, $\alpha$ of type (A), $\beta$ of type (A).\\
\hline 22 & 6 & $G=C_6=\langle\sigma\rangle$, $\sigma$ of type (B1).\\
\hline 19 & 6 & $G=C_6=\langle\sigma\rangle$, $\sigma$ of type (B1).\\
\hline 16 & 6 & $G=C_6=\langle\sigma\rangle$, $\sigma$ of type (A).\\
\hline 14 & 9 & $G=C_9=\langle\sigma\rangle$, $\sigma$ of type (B1).\\
\hline 12 & 8 & $G=D_8=C_4 \rtimes C_2=\langle\alpha\rangle \rtimes \langle \beta \rangle$, $\alpha$ of type (B2), $\beta$ of type (A).\\
\hline 11 & 8 & $G=Dic_{12}$.\\
\hline 10 & 12 & $G=A_4$.\\
\hline 8 & 9 & $G=C_9=\langle\sigma\rangle$, $\sigma$ of type (A).\\
\hline 7 & 6 & $G=C_6=\langle\sigma\rangle$, $\sigma$ of type (A).\\
\hline 6 & 18 & $G=C_{18}=\langle\sigma\rangle$, $\sigma$ of type (B1).\\
\hline 5 & 18 & $G=C_6 \times C_3=\langle\alpha\rangle \times \langle \beta \rangle$, $\alpha$ of type (A), $\beta$ of type (B1).\\
\hline 4 & 18 & $G=C_6 \times C_3=\langle\alpha\rangle \times \langle \beta \rangle$, $\alpha$ of type (A), $\beta$ of type (A).\\
\hline 3 & 18 & $G=Sym(3) \times C_3$.\\
\hline 2 & 18 & $G=C_{3} \times C_3 \rtimes C_2 = \langle \alpha \rangle \times \langle \beta \rangle \times \langle \gamma \rangle$, $\alpha$, $\beta$, $\gamma$ of type (A). \\
\hline 1 & 91 & $G=C_{91}=\langle\sigma\rangle$, $\sigma$ of type (B3). \\
\hline 0 & 18 & $G=C_{18}=\langle\sigma\rangle$, $\sigma$ of type (A).\\
\hline
\end{tabular}
\end{small}
\end{table}
\end{center}

%\textcolor{red}{per $q=17$ i turchi hanno scordato $g=5$ ma viene da GSX!}

\begin{center}
\begin{table}[htbp]
\begin{small}
\caption{Quotient curves $\cH_{19}/G$, \ $G \leq \PGU(3,19)$}\label{tabella12}
\begin{tabular}{|c|c|c|}
\hline \bf $g$ & $|G|$ & \textrm{structure of $G$} \\
\hline\hline 171 & 1 & \textrm{trivial group.}\\
\hline 81 & 2 & $G=C_2=\langle\sigma\rangle$, $\sigma$ of type (A).\\
\hline 57 & 3 & $G=C_3=\langle\sigma\rangle$, $\sigma$ of type (B2).\\
\hline 41 & 4 & $G=C_4=\langle\sigma\rangle$, $\sigma$ of type (B1).\\
\hline 36 & 4 & $G=C_4=\langle\sigma\rangle$, $\sigma$ of type (A).\\
\hline 35 & 5 & $G=C_5=\langle\sigma\rangle$, $\sigma$ of type (B1).\\
\hline 27 & 4 & $G=C_4=\langle\sigma\rangle$, $\sigma$ of type (A).\\
\hline 24 & 6 & $G=Sym(3)=C_3 \rtimes C_2=\langle \alpha \rangle \rtimes \langle \beta \rangle$, $\alpha$ of type (B2), $\beta$ of type (A).\\
\hline 21 & 6 & $G=Q_8$.\\
\hline 19 & 9 & $G=C_9=\langle\sigma\rangle$, $\sigma$ of type (B2).\\
\hline 18 & 8 & $G=C_8=\langle\sigma\rangle$, $\sigma$ of type (B2).\\
\hline 17 & 10 & $G=C_{10}=\langle\sigma\rangle$, $\sigma$ of type (B1).\\
\hline 16 & 8 & $G=D_8=C_4 \rtimes C_2=\langle\alpha\rangle \rtimes \langle \beta \rangle$, $\alpha$ of type (B1), $\beta$ of type (A).\\
\hline 14 & 12 & $G=Dic_{12}$.\\
\hline 13 & 10 & $G=C_{10}=\langle\sigma\rangle$, $\sigma$ of type (B1).\\
\hline 12 & 12 & $G=C_{12}=\langle\sigma\rangle$, $\sigma$ of type (B2).\\
\hline 9 & 10 & $G=C_{10}=\langle\sigma\rangle$, $\sigma$ of type (A).\\
\hline 8 & 21 & $G=C_7 \rtimes C_3=\langle\alpha\rangle \rtimes \langle \beta \rangle$, $\alpha$ of type (B3), $\beta$ of type (B2).\\
\hline 7 & 24 & $G=SL(2,3)$.\\
\hline 6 & 24 & $G=C_3 \rtimes C_8=\langle\alpha\rangle \rtimes \langle \beta \rangle$, $\alpha$ of type (B2), $\beta$ of type (B2).\\
\hline 5 & 18 & $G=D_9=C_9 \rtimes C_2=\langle\alpha\rangle \rtimes \langle \beta \rangle$, $\alpha$ of type (B2), $\beta$ of type (A).\\
\hline 4 & 24 & $G=Sym(3)\times C_4$.\\
\hline 3 & 30 & $G=C_{30}=\langle\sigma\rangle$, $\sigma$ of type (B2).\\
\hline 2 & 18 & $G=SG(32,11)$. \\
\hline 1 & 141 & $G=C_{49} \rtimes C_3=\langle\alpha\rangle \rtimes \langle \beta \rangle$, $\alpha$ of type (B3), $\beta$ of type (B2). \\
\hline 0 & 20 & $G=C_{20}=\langle\sigma\rangle$, $\sigma$ of type (A).\\
\hline
\end{tabular}
\end{small}
\end{table}
\end{center}

\begin{center}
\begin{table}[htbp]
\begin{small}
\caption{Quotient curves $\cH_{23}/G$, \ $G \leq \PGU(3,23)$}\label{tabella13}
\begin{tabular}{|c|c|c|}
\hline \bf $g$ & $|G|$ & \textrm{structure of $G$} \\
\hline\hline 253 & 1 & \textrm{trivial group.}\\
\hline 121 & 2 & $G=C_2=\langle\sigma\rangle$, $\sigma$ of type (A).\\

\hline 85 & 3 & $G=C_3=\langle\sigma\rangle$, $\sigma$ of type (B1).\\
\hline 84 & 3 & $G=C_3=\langle\sigma\rangle$, $\sigma$ of type (B3).\\
\hline 77 & 3 & $G=C_3=\langle\sigma\rangle$, $\sigma$ of type (A).\\

\hline 61 & 4 & $G=C_4=\langle\sigma\rangle$, $\sigma$ of type (B1).\\
\hline 55 & 4 & $G=C_4=\langle\sigma\rangle$, $\sigma$ of type (A).\\

\hline 41 & 6 & $G=C_6=\langle\sigma\rangle$, $\sigma$ of type (B1).\\
\hline 37 & 6 & $G=Sym(3)=C_3 \rtimes C_2=\langle \alpha \rangle \rtimes \langle \beta \rangle$, $\alpha$ of type (B1), $\beta$ of type (A).\\
\hline 33 & 6 & $G=C_6=\langle\sigma\rangle$, $\sigma$ of type (A).\\

\hline 31 & 8 & $G=Q_8$.\\
\hline 28 & 8 & $G=C_8=\langle\sigma\rangle$, $\sigma$ of type (B1).\\
\hline 25 & 8 & $G=C_4 \times C_2=\langle\alpha\rangle \times \langle \beta \rangle$, $\alpha$ of type (A), $\beta$ of type (A).\\
\hline 23 & 11 & $G=C_{11}=\langle\sigma\rangle$, $\sigma$ of type (B2).\\
\hline 22 & 8 & $G=C_8=\langle\sigma\rangle$, $\sigma$ of type (A).\\

\hline 21 & 12 & $G=C_{12}=\langle\sigma\rangle$, $\sigma$ of type (B1).\\
\hline 19 & 12 & $G=C_{12}=\langle\sigma\rangle$, $\sigma$ of type (B1).\\
\hline 17 & 12 & $G=C_{12}=\langle\sigma\rangle$, $\sigma$ of type (B1).\\
\hline 16 & 16 & $G=Q_{16}$.\\
\hline 15 & 12 & $G=C_6 \times C_2=\langle\alpha\rangle \times \langle \beta \rangle$, $\alpha$ of type (A), $\beta$ of type (A).\\
\hline 13 & 16 & $G=C_8 \times C_2=\langle\alpha\rangle \times \langle \beta \rangle$, $\alpha$ of type (B1), $\beta$ of type (A).\\
\hline 11 & 12 & $G=C_{12}=\langle\sigma\rangle$, $\sigma$ of type (A).\\
\hline 10 & 16 & $G=C_4 \times C_4=\langle\alpha\rangle \times \langle \beta \rangle$, $\alpha$ of type (A), $\beta$ of type (A).\\
\hline 9 & 18 & $G=C_6 \times C_3=\langle\alpha\rangle \times \langle \beta \rangle$, $\alpha$ of type (A), $\beta$ of type (A).\\
\hline 8 & 24 & $G=C_{24}=\langle\sigma\rangle$, $\sigma$ of type (B1).\\
\hline 7 & 33 & $G=C_{33}=\langle\sigma\rangle$, $\sigma$ of type (B2).\\
\hline 6 & 22 & $G=C_{11} \rtimes C_2=\langle\alpha\rangle \times \langle \beta \rangle$, $\alpha$ of type (B2), $\beta$ of type (A).\\
\hline 5 & 18 & $G=C_3 \times C_3 \rtimes C_2=\langle\alpha\rangle \times \langle \beta \rangle \rtimes \langle \gamma \rangle$, $\alpha$, $\beta$ of type (A), $\gamma$ of type (A) .\\
\hline 4 & 32 & $G=C_{8} \times C_4=\langle\alpha\rangle \times \langle \beta \rangle$, $\alpha$ of type (A), $\beta$ of type (A).\\
\hline 3 & 24 & $G=D_8 \times C_3=(C_4 \rtimes C_2) \times C_3=(\langle\alpha\rangle \rtimes \langle \beta \rangle) \times \langle \gamma \rangle$, $\alpha$ of type (B1), $\beta$, $\gamma$ of type (A).\\
\hline 2 & 88 & $G=C_{88}=\langle\sigma\rangle$, $\sigma$ of type (B2).\\
\hline 1 & 169 & $G=C_{169}=\langle\sigma\rangle$, $\sigma$ of type (B3).\\
\hline 0 & 24 & $G=C_{24}=\langle\sigma\rangle$, $\sigma$ of type (A).\\
\hline
\end{tabular}
\end{small}
\end{table}
\end{center}

\begin{center}
\begin{table}[htbp]
\begin{small}
\caption{Quotient curves $\cH_{25}/G$, \ $G \leq \PGU(3,25)$}\label{tabella14}
\begin{tabular}{|c|c|c|}
\hline \bf $g$ & $|G|$ & \textrm{structure of $G$} \\
\hline\hline 300 & 1 & \textrm{trivial group.}\\
\hline 144 & 2 & $G=C_2=\langle\sigma\rangle$, $\sigma$ of type (A).\\

\hline 100 & 3 & $G=C_3=\langle\sigma\rangle$, $\sigma$ of type (B2).\\

\hline 72 & 4 & $G=C_4=\langle\sigma\rangle$, $\sigma$ of type (B2).\\
\hline 66 & 4 & $G=C_2 \times C_2=\langle\alpha\rangle \times \langle \beta \rangle$, $\alpha$ of type (A), $\beta$ of type (A).\\

\hline 60 & 5 & $G=C_5=\langle\sigma\rangle$, $\sigma$ of type (D).\\
\hline 50 & 5 & $G=C_5=\langle\sigma\rangle$, $\sigma$ of type (C).\\

\hline 48 & 6 & $G=C_6=\langle\sigma\rangle$, $\sigma$ of type (B2).\\
\hline 44 & 6 & $G=Sym(3)=C_3 \rtimes C_2=\langle \alpha \rangle \rtimes \langle \beta \rangle$, $\alpha$ of type (B2), $\beta$ of type (A).\\

\hline 36 & 8 & $G=Q_8$.\\
\hline 30 & 8 & $G=D_8=C_4 \rtimes C_2=\langle\alpha\rangle \rtimes \langle \beta \rangle$, $\alpha$ of type (B2), $\beta$ of type (A).\\

\hline 24 & 10 & $G=C_{10}=\langle\sigma\rangle$, $\sigma$ of type (E).\\

\hline 22 & 12 & $G=C_6 \times C_2=\langle\alpha\rangle \times \langle \beta \rangle$, $\alpha$ of type (B2), $\beta$ of type (A).\\
\hline 18 & 12 & $G=D_{12}=C_6 \rtimes C_2=\langle\alpha\rangle \rtimes \langle \beta \rangle$, $\alpha$ of type (B2), $\beta$ of type (A).\\

\hline 12 & 13 & $G=C_{13}=\langle\sigma\rangle$, $\sigma$ of type (A).\\
\hline 10 & 25 & $G=C_{5} \times C_5=\langle\alpha\rangle \times \langle \beta \rangle$, $\alpha$ of type (C), $\beta$ of type (D).\\

\hline 8 & 39 & $G=C_{13} \rtimes C_3=\langle\alpha\rangle \rtimes \langle \beta \rangle$, $\alpha$ of type (B1), $\beta$ of type (B2).\\

\hline 6 & 24 & $G=D_{24}=C_{12} \rtimes C_2=\langle\alpha\rangle \rtimes \langle \beta \rangle$, $\alpha$ of type (B2), $\beta$ of type (A).\\

\hline 4 & 39 & $G=C_{39}=\langle\sigma\rangle$, $\sigma$ of type (B2).\\

\hline 3 & 52 & $G=C_{13} \rtimes C_4=\langle\alpha\rangle \rtimes \langle \beta \rangle$, $\alpha$ of type (B1), $\beta$ of type (B2).\\

\hline 2 & 125 & $G=C_{5} \times C_5 \times C_5=\langle\alpha\rangle \times \langle \beta \rangle \times \langle \gamma \rangle$, $\alpha$ of type (C), $\beta$, $\gamma$ of type (D).\\

\hline 0 & 26 & $G=C_{26}=\langle\sigma\rangle$, $\sigma$ of type (A).\\
\hline
\end{tabular}
\end{small}
\end{table}
\end{center}

\begin{center}
\begin{table}[htbp]
\begin{small}
\caption{Quotient curves $\cH_{27}/G$, $G \leq \PGU(3,27)$ }\label{tabella15}
\begin{tabular}{|c|c|c|}
\hline \bf $g$ & $|G|$ & \textrm{structure of $G$} \\
\hline\hline 351 & 1 & \textrm{trivial group.}\\
\hline 169 & 2 & $G=C_2=\langle\sigma\rangle$, $\sigma$ of type (A).\\
\hline 117 & 3 & $G=C_3=\langle\sigma\rangle$, $\sigma$ of type (D).\\
\hline 108 & 3 & $G=C_3=\langle\sigma\rangle$, $\sigma$ of type (C).\\
\hline 85 & 4 & $G=C_4=\langle\sigma\rangle$, $\sigma$ of type (B1).\\
\hline 78 & 4 & $G=C_4=\langle\sigma\rangle$, $\sigma$ of type (A).\\
\hline 52 & 6 & $G=C_6=\langle\sigma\rangle$, $\sigma$ of type (E).\\
  &  & $G=Sym(3)=\langle\alpha\rangle\rtimes\langle\beta\rangle$, $\alpha$ of type (C), $\beta$ of type (A).\\
\hline 51 & 7 & $G=C_7=\langle\sigma\rangle$, $\sigma$ of type (B1).\\
\hline 43 & 8 & $G= Q_8$ quaternion group, 1 element of type (A), 6 elements of type (B1).\\
\hline 39 & 7 & $G=C_7=\langle\sigma\rangle$, $\sigma$ of type (A).\\
\hline 27 & 9 & $G=C_3\times C_3=\langle\alpha\rangle\times\langle\beta\rangle$, $\alpha$ and $\beta$ of type (C).\\
\hline 26 & 12 & $G= Alt(4)$, involutions of type (A), other elements of type (D).\\
\hline 25 & 14 & $G=C_{14}=\langle\sigma\rangle$, $\sigma$ of type (B1).\\
\hline 24 & 12 & $G=C_{12}=\langle\sigma\rangle$, $\sigma$ of type (E).\\
\hline 19 & 14 & $G=C_{14}=\langle\sigma\rangle$, $\sigma$ of type (B1), $\sigma^2$ of type (A).\\
\hline 18 & 16 & $G= M_{16}$, 5 elements of type (A),\\ & & 2 elements of type (B1), 8 elements of type (B2).\\
\hline 18 & 19 & $G=C_{19}=\langle\sigma\rangle$, $\sigma$ of type (B3).\\
\hline 17 & 21 & $G=C_7\rtimes C_3=\langle\alpha\rangle\rtimes\langle\beta\rangle$, $\alpha$ of type (B1), $\beta$ of type (B2).\\
\hline 16 & 18 & $G=C_3\times (C_3\rtimes C_2)=\langle\alpha\rangle\times(\langle\beta\rangle\rtimes\langle\gamma\rangle)$,\\ & & $\alpha$ of type (C), $\beta$ of type (D), $\gamma$ of type (A).\\
\hline 15  & 16 & $G=D_8\circ C_4=(\langle\alpha\rangle\rtimes\langle\beta\rangle)\circ\langle\gamma\rangle$, $\alpha$ of type (B1), $\beta$ and $\gamma$ of type (A).\\
\hline 13 & 14 & $G=C_{14}=\langle\sigma\rangle$, $\sigma$ of type (A).\\
\hline 12 & 21 & $G=C_{21}=\langle\sigma\rangle$, $\sigma$ of type (E).\\
\hline 10 & 24 & $G\cong{\rm SL}(2,3)$, 1 element of type (A), 6 elements of type (B1),\\ & & 8 elements of type (C), 8 elements of type (E).\\
\hline 9 & 37 & $G=C_{37}=\langle\sigma\rangle$, $\sigma$ of type (B3).\\
\hline 7 & 26 & $G=C_{13}\rtimes C_2=\langle\alpha\rangle\rtimes\langle\beta\rangle$, $\alpha$ of type (B2), $\beta$ of type (A).\\
\hline 6 & 32 & $G=C_4\wr C_2 = \langle\alpha\rangle\wr\langle\beta\rangle$ wreath product,\\ & & 13 elements of type (A), 10 elements of type (B1), 8 elements of type (B2).\\
\hline 5 & 48 & $G=(C_4\times C_4)\rtimes C_3=(\la\alpha\ra\T\la\beta\ra)\RT\la\gamma\ra$, $\alpha$ and $\beta$ of type (A), $\gamma$ of type (D).\\
\hline 4 & 48 & $G=(D_8\circ C_4)\RT\la\sigma\ra$, $\sigma$ of type (C).\\
\hline 3 & 49 & $G=C_7\times C_7=\langle\alpha\rangle\times\langle\beta\rangle$, $\alpha$ and $\beta$ of type (A).\\
\hline 1 & 126 & $G=C_{14}\times C_3 \times C_3=\langle\alpha\rangle\times\langle\beta\rangle \times \langle \gamma \rangle$, $\alpha$ of type (A), $\beta$, $\gamma$ of type (C).\\
\hline 0 & 28 & $G=C_{28}=\langle\sigma\rangle$, $\sigma$ of type (A).\\
\hline
\end{tabular}
\end{small}
\end{table}
\end{center}

\begin{center}
\begin{table}[htbp]
\begin{small}
\caption{Quotient curves $\cH_{29}/G$, $G \leq \PGU(3,29)$ }\label{tabella16}
\begin{tabular}{|c|c|c|}
\hline \bf $g$ & $|G|$ & \textrm{structure of $G$} \\
\hline\hline 406 & 1 & \textrm{trivial group.}\\
\hline 196 & 2 & $G=C_2=\langle\sigma\rangle$, $\sigma$ of type (A).\\

\hline 136 & 3 & $G=C_3=\langle\sigma\rangle$, $\sigma$ of type (B1).\\
\hline 135 & 3 & $G=C_3=\langle\sigma\rangle$, $\sigma$ of type (B3).\\
\hline 126 & 3 & $G=C_3=\langle\sigma\rangle$, $\sigma$ of type (A).\\
\hline 98 & 4 & $G=C_4=\langle\sigma\rangle$, $\sigma$ of type (B2).\\
\hline 91 & 4 & $G=C_2 \times C_2=\langle\alpha\rangle \times \langle \beta \rangle $, $\alpha$ and $\beta$ of type (A).\\

\hline 82 & 5 & $G=C_5=\langle\sigma\rangle$, $\sigma$ of type (B1).\\
\hline 70 & 5 & $G=C_5=\langle\sigma\rangle$, $\sigma$ of type (A).\\

\hline 66 & 6 & $G=C_6=\langle\sigma\rangle$, $\sigma$ of type (B1).\\
\hline 61  & 6 & $G=Sym(3)=\langle\alpha\rangle\rtimes\langle\beta\rangle$, $\alpha$ of type (B1), $\beta$ of type (A).\\
\hline 58 & 7 & $G=C_7=\langle\sigma\rangle$, $\sigma$ of type (B2).\\
\hline 56 & 6 & $G=C_6=\langle\sigma\rangle$, $\sigma$ of type (A).\\
%\hline 49 & 8 & $G=C_8= \langle \alpha \rangle$, $\alpha$ of type (B2).\\
\hline 49 & 8 & $G=Q_8=\langle\alpha,\beta\rangle$, $\alpha$ and $\beta$ of type (B2).\\
\hline 45 & 9 & $G=C_3\times C_3=\langle\alpha\rangle\times\langle\beta\rangle$, $\alpha$ of type (B3), $\beta$ of type (B1).\\
\hline 42 & 8 & $G=D_8=C_4\rtimes C_2=\langle\alpha\rangle\rtimes\langle\beta\rangle$, $\alpha$ of type (B2), $\beta$ of type (A).\\
\hline 40 & 10 & $G=C_{10}=\langle\sigma\rangle$, $\sigma$ of type (B1).\\
\hline 36 & 9 & $G=C_3\times C_3=\langle\alpha\rangle\times\langle\beta\rangle$, $\alpha$ and $\beta$ of type (A).\\
\hline 34 & 10 & $G=C_5 \rtimes C_2=\langle\alpha\rangle\rtimes\langle\beta\rangle$, $\alpha$ of type (B1), $\beta$ of type (A).\\
\hline 33 & 12 & $G=Dic_{12}=\langle\alpha,\beta\rangle$, $\alpha$ of type (B1), $\beta$ of type (B2).\\
\hline 31 & 12 & $G=Alt(4) = \langle\alpha\rangle\times\langle\beta\rangle\rtimes\langle\gamma\rangle$, $\alpha$ and $\beta$ of type (A), $\gamma$ of type (B1).\\
\hline 28 & 10 & $G=C_{10}=\langle\sigma\rangle$, $\sigma$ of type (A).\\
\hline 26 & 12 & $G=D_{12}=C_6 \rtimes C_2=\langle\alpha\rangle\rtimes\langle\beta\rangle$, $\alpha$ of type (B1), $\beta$ of type (A).\\
\hline 24 & 15 & $G=C_{15}=\langle\sigma\rangle$, $\sigma$ of type (B1).\\
\hline 22 & 14 & $G=D_{14}=C_7 \rtimes C_2=\langle\alpha\rangle\rtimes\langle\beta\rangle$, $\alpha$ of type (B2), $\beta$ of type (A).\\
\hline 21 & 16 & $G=SD_{16}$, semidihedral group (Sylow $2$-subgroup of $\PGU(3,29)$).\\
\hline 20 & 20& $D=Dic_{20}$. \\
\hline 19 & 20 & $G=C_{10} \times C_2=\langle\alpha\rangle\times\langle\beta\rangle$, $\alpha$ of type (B1), $\beta$ of type (A).\\
\hline 18 & 21 & $G=C_{21}=\langle\sigma\rangle$, $\sigma$ of type (B2).\\
\hline 17 & 24 & $G=SL(2,3)$. \\
\hline 16 & 18 & $G=C_6 \times C_3 = \langle \alpha \rangle \times \langle \beta \rangle$, $\alpha$ and $\beta$ of type (A). \\
\hline 14 & 15 & $G=C_{15}=\langle\sigma\rangle$, $\sigma$ of type (A).\\
\hline 13 & 20 & $G=D_{20}=C_{10} \rtimes C_2=\langle\alpha\rangle\rtimes\langle\beta\rangle$, $\alpha$ of type (B1), $\beta$ of type (A).\\
\hline 12 & 24 & $G=D_8\times C_3=(\langle\alpha\rangle\rtimes\langle\beta\rangle)\times\langle\gamma\rangle$, $\alpha$ of type (B2), $\beta$ and $\gamma$ of type (A). \\
%$G=Dic_{12} \rtimes C_2.$ \\
\hline 11 & 30 & $G=D_{30}=C_{15}\rtimes C_2=\langle\alpha\rangle\rtimes\langle\beta\rangle$, $\alpha$ of type (B1), $\beta$ of type (A). \\
\hline 10 & 25 & $G=C_5 \times C_5 = \langle \alpha \rangle \times \langle \beta \rangle$, $\alpha$, $\beta$ of type (A).\\
\hline 8 & 36 & $G=(C_3\times C_3)\rtimes C_4= (\langle\alpha\rangle\times \langle\beta\rangle)\rtimes \langle\gamma\rangle$, $\alpha,\beta$ of type (A), $\gamma$ of type (B2).\\
\hline 7 & 24 & $G=D_{24}=C_{12} \rtimes C_2=\langle\alpha\rangle\rtimes\langle\beta\rangle$, $\alpha$ of type (B2), $\beta$ of type (A).\\
\hline 6 & 36 & $G=C_3 \times C_3 \times C_2 \times C_2.$ \\
\hline 5 & 48 & $G=SD_{16}\rtimes C_3$.\\
\hline 4 & 45 & $G=C_{15} \times C_3=\langle\alpha\rangle\times\langle\beta\rangle$, $\alpha$ and $\beta$ of type (A).\\
\hline 3 & 120 & $G=Q_8 \rtimes C_{15}=Q_8 \rtimes \langle \alpha \rangle $,$\alpha$ of type (B1).\\
\hline 2 & 42 & $G=D_{42}=C_{21} \rtimes C_2=\langle\alpha\rangle\rtimes\langle\beta\rangle$, $\alpha$ of type (B2), $\beta$ of type (A).\\
\hline 1 & 271 & $G=C_{271}=\langle\sigma\rangle$, $\sigma$ of type (B3).\\
\hline 0 & 30 & $G=C_{30}=\langle\sigma\rangle$, $\sigma$ of type (A).\\
\hline
\end{tabular}
\end{small}
\end{table}
\end{center}

\section{New genera for maximal curves}

The results of the previous sections provide new genera for maximal curves over finite fields. To exemplify this fact, we collect in this section new genera for small values of $q$.

\begin{remark}
Tables \ref{tabella9} and \ref{tabella10} partially answer to \cite[Remark 4.4]{ATT}.
In particular, we have obtained a $\mathbb F_{13^2}$-maximal curve of genus $4$ from Prop. \ref{Esiste1}, a $\mathbb F_{13^2}$-maximal curve of genus $5$ from Prop. \ref{Esiste2}, and a $\mathbb F_{16^2}$-maximal curve of genus $16$ from Prop. \ref{Esiste3}.
%Note also that \cite[Example 5.12]{GSX} provides an example of $\mathbb F_{13^2}$-maximal curve, hence answering a question in \cite{Remark 4.4}{ATT}.
\end{remark}

\begin{remark}
The value $g=3$ in Table \ref{tabella12} is new in the spectrum of genera of $\mathbb{F}_{13^2}$-maximal curves. In fact this value is given in \cite{DO2} as an application of \cite[Proposition 4.6]{GSX}, but the actual value which provided by \cite[Proposition 4.6]{GSX} is $g=0$. Moreover, it can be checked that if $g(\cH_{13}/G)=3$ then $G \leq \mathcal{M}_{13}$.
\end{remark}

\begin{remark}
The genera $g=3$, $g=8$ and $g=22$ provided in Table \ref{tabella17} for $\mathbb{F}_{25^2}$-maximal curves are new with respect to \cite{DO2}.
\end{remark}

%\textcolor{red}{
%guardiamo i nuovi generi solo in $3^4,5^4,7^2,7^4,11^2,13^2,17^2,11^4,23^2$, che sono favorevoli perche non ci sono GK, GGS, Ree, Suzuki; guardare solo turchi e CKT2 Sezione 3.}

\end{document}